\documentclass[12pt]{elsarticle}

\usepackage{amsmath, amssymb, amsthm}
\usepackage{hyperref}
\usepackage{tabularx}
\usepackage{mathabx} 
\usepackage{enumitem}
\usepackage{lineno}
\usepackage{tikz}
\usetikzlibrary{arrows}
\tikzset{
every node/.style={circle, draw, inner sep=2pt}, 
every label/.style={rectangle, draw=none}
}

\usepackage{soul}

\newcommand{\bI}{I}
\newcommand{\bA}{A}

\newcommand{\calS}{\mathcal{S}}

\newcommand{\zr}[1]{Z^{(#1)}}
\newcommand{\zpr}[1]{Z_+^{(#1)}}
\newcommand{\nul}{\operatorname{null}}

\newcommand{\confl}{\operatorname{conf}_{\ell}}
\newcommand{\confe}{\operatorname{conf}_{\rm e}}
\newcommand{\Zhat}{\widehat{Z}}

\newcommand{\Zcheck}{\widecheck{Z}}
\newcommand{\Zpcheck}{\widecheck{Z}_+}
\newcommand{\trans}{^\top}
\newcommand{\oml}{{\bf m}}
\newcommand{\diam}{\operatorname{diam}}
\newcommand{\TT}{W} 
\newcommand{\rank}{\operatorname{rank}}
\newcommand{\bzero}{{\bf 0}}
\newcommand{\ba}{{\bf a}}

\newcommand{\gloop}{G^{\rm loop}}
\newcommand{\tloop}{T^{\rm loop}}
\newcommand{\gO}{G^{\rm 0}}

\newtheorem{theorem}{Theorem}[section]
\newtheorem{lemma}[theorem]{Lemma}
\newtheorem{corollary}[theorem]{Corollary}
\newtheorem{proposition}[theorem]{Proposition}

\theoremstyle{definition}
\newtheorem{definition}[theorem]{Definition}
\newtheorem{example}[theorem]{Example}

\newtheorem{remark}[theorem]{Remark}
\newcommand{\ii}{{\rm i}}
\newcommand{\bx}{{\bf x}}
\newcommand{\calskew}{\mathcal{S}_-(G)}
\newcommand{\im}{\operatorname{Im}}

\begin{document}

\begin{frontmatter}

\title{A zero forcing technique for bounding sums of eigenvalue multiplicities}  



            

\author[Kenteraddress]{Franklin H.J. Kenter}
\address[Kenteraddress]{United States Naval Academy, Mathematics Department, Annapolis, MD, 21401, U.S.A.}
\ead{kenter@usna.edu}

\author[Linaddress]{Jephian C.-H. Lin\corref{corres}}
\address[Linaddress]{Department of Applied Mathematics, National Sun Yat-sen University, Kaohsiung 80424, Taiwan}
\cortext[corres]{Corresponding author}
\ead{jephianlin@gmail.com}

\begin{abstract}
Given a graph $G$, one may ask: ``What sets of eigenvalues are possible over all weighted adjacency matrices of $G$?'' (The weight of an edge is positive or negative, while the diagonal entries can be any real numbers.)  This is known as the Inverse Eigenvalue Problem for graphs (IEP-$G$).  A mild relaxation of this question considers the multiplicity list instead of the exact eigenvalues themselves. That is, given a graph $G$ on $n$ vertices and an ordered partition $\oml= (m_1, \ldots, m_\ell)$ of $n$, is there a weighted adjacency matrix where the $i$-th distinct eigenvalue has multiplicity $m_i$?  This is known as the ordered multiplicity IEP-$G$.
Recent work solved the ordered multiplicity IEP-$G$ for all graphs on 6 vertices. 

In this work, we develop zero forcing methods for the ordered multiplicity IEP-$G$ in a multitude of different contexts. Namely, we utilize zero forcing parameters on powers of graphs to achieve bounds on consecutive multiplicities. We are able to provide general bounds on sums of multiplicities of eigenvalues for graphs. This includes new bounds on the the sums of multiplicities of consecutive eigenvalues as well as more specific bounds for trees. Using these results, we verify the previous results above regarding the IEP-$G$ on six vertices. In addition, applying our techniques to skew-symmetric matrices, we are able to determine all possible ordered multiplicity lists for skew-symmetric matrices for connected graphs on five vertices.
\end{abstract}

\begin{keyword}
 inverse eigenvalue problem for graphs (IEP-$G$),
 ordered multiplicity sequences, 
 powers of graphs, 
 skew-symmetric matrices, 
 zero forcing
\MSC{05C50, 05C57, 15A29, 15A42, 90C10}
\end{keyword}

\end{frontmatter}

\section{Introduction}
Given a graph $G$, the inverse eigenvalue problem asks the question: ``What (multi-)sets of eigenvalues are possible over all weighted adjacency matrices of $G$?''  Here, an edge weight is a nonzero value (positive or negative) and the diagonal entries can be any real number.

Zero forcing is a one-player game played on a graph whereby the player colors an initial set of vertices, then applies a propagation process during which colored  vertices may force uncolored vertices. The goal is to find minimum set of vertices such that eventually all of the vertices become colored. 
Many variations of zero forcing are used to bound the maximum nullity over certain classes of matrices associated with $G$. For instance, the original variation of the game was introduced as the result of an AIM workshop \cite{zf_aim} and helped determine the maximum nullity for symmetric graphs (allowing weighted diagonals) for all graphs of up to 7 vertices \cite{small}. Since then, a multitude of variations of zero forcing have been developed for other classes of graphs including but certainly not limited to skew-symmetric matrices \cite{zf_skew}, sign patterns \cite{zf_sign}, hypermatrices \cite{leslie_hyper}, positive semidefinite matrices \cite{zf_psd}, matrices with limited negative eigenvalues \cite{zf_q}, multigraphs \cite{cancun}, looped graphs \cite{square1, treewidth}, and of course, combinations of these cases \cite{zf_psd_multi}. Additionally, in some cases, refinements of these methods have been made by introducing additional rules such as odd-cycle conditions \cite{Zoc}. Indeed, zero forcing has proven such a popular topic in its own right that it has spawned variations that remove the linear-algebraic context altogether such as $k$-forcing \cite{kforcing}. Specific applications such as power domination \cite{powerdomination} have spawned their own lines of research as well.  Table \ref{tbl:summary} summarizes a few variations of zero forcing. 

Many of these variations live in isolation and only apply when the corresponding class of matrices arises or within a specific application. In this article, we demonstrate that many of these variations, when considered jointly, can help paint a much clearer, if not definitive picture, as to what eigenvalues are possible under a variety of different constraints.

The original motivation of this study was to study ordered eigenvalue multiplicity lists. An ordered eigenvalue multiplicity list for a matrix is a list $(m_1, \ldots, m_\ell)$ such that the $i$-th distinct eigenvalue has precisely multiplicity $m_i$. In \cite{ordered_iepg}, the authors tirelessly classify all allowable ordered eigenvalues multiplicity lists over all weighted adjacency matrices (with arbitrary diagonal) for all graphs up to 6 vertices.  (The cases for graphs up to $5$ vertices were done in \cite{BNSY14}, while cases for graphs up to $4$ vertices were done in \cite{minormult}.) Most of the cases therein are covered using a variety of parameters, including the zero forcing number, or specialized results. However, six exceptional graphs required significant additional analysis. In contrast, we develop a robust computational approach, using {\it just} the zero forcing numbers to validate the results in \cite{ordered_iepg}. Our results arrive at similar, but fewer, exceptional cases.

Our main approach is to 
apply a myriad of different, straight-forward, zero forcing parameters in order to exclude the possibility of certain multiplicity lists. In addition to using previously developed zero forcing parameters
we will repeatedly make use of 
combinations of zero forcing techniques not widely used before with a focus on powers of graphs. These parameters will provide an upper bound for sums of various elements in the multiplicity list. With these bounds, we construct a system of linear constraints in order to determine the region of feasibility which provides candidates for allowable ordered eigenvalues multiplicity lists. 

A zero forcing parameter, rigid linkage forcing, was recently introduced in \cite{rigid} to bound the total multiplicity of multiple eigenvalues. While our approach has the same goal of studying multiplicities of eigenvalues through zero forcing, they do not appear comparable or related. Though, one advantage of our approach is that it is more straight-forward to implement on a large scale as we do in Section \ref{subsec:six}.

Using these methods, we will be able recreate several previously known linear-algebraic results as well as new variations of these results solely using zero forcing parameters. Among these include:

\begin{itemize}
    \item Developing zero forcing processes on powers of graphs and relating these parameters to multiplicity lists (Section \ref{sec:power}, Theorems \ref{thm:main} and \ref{thm:mainzp}).
    \item Providing a uniform bound for multiplicities of eigenvalues for trees (Theorem \ref{thm:tree}).
    \item Developing an improved bound for the minimum number of distinct eigenvalues for a graph, $q(G)$, using zero forcing parameters (Theorem \ref{thm:qbound}).
     \item Providing a new argument, using only zero forcing parameters, that the tree in \cite[Fig~3.1]{barioli2004two} requires a number of distinct eigenvalues two more than its diameter (Theorem \ref{thm:bftree}).
    \item  Verifying that all ordered eigenvalues multiplicity lists for graphs with at most 6 vertices that are not listed in \cite{ordered_iepg} are indeed not possible (Subsection \ref{subsec:six}).
    \item Adapting the techniques from Section \ref{sec:power} to multiplicity lists for skew-symmetric matrices (Section \ref{sec:skew}).
    \item Determining all realizable multiplicity lists for connected graphs on 5 vertices and providing realizations for each  (\ref{sec:skewrealize}).
\end{itemize}

\section{Preliminaries}

We focus on studying finite, undirected graphs. However, in doing so, we may allow a graph to have multiple edges, a \emph{multigraph}; or have loops, a \emph{looped graph}; or both, \emph{a looped multigraph}. We will call a graph \emph{simple} if it is neither a multigraph nor a looped graph. A \emph{general graph} is a graph, looped graph, multigraph or looped multigraph.

For general graphs, we use the notation $i \sim j$ to denote that vertex $i$ is adjacent to vertex $j$. In the case of looped graphs, $i \sim i$ denotes a loop at $i$. For multigraphs, $i \sim_! j$ denotes that there is exactly one edge between $i$ and $j$; we call such an edge a \emph{singleton edge}. The \emph{underlying graph} of a general graph is a simple graph formed by removing all loops and/or removing all but one edge between every pair of adjacent vertices. 

Given a simple graph $G$, we define $\calS(G)$ to be the set of all $n \times n$ real symmetric matrices whereby the $ij$-entry, $i\neq j$, is nonzero whenever $i \sim j$ and zero otherwise.  The diagonal may be any combination of zero or nonzero entries. 

If $G$ is a looped graph, then $\calS(G)$ is defined to be the set of all $n \times n$ symmetric matrices whereby the $ij$-entry is nonzero whenever $i \sim j$ and the $ii$ diagonal entry must zero if there is no loop and must be nonzero if there is a loop at $i$. 

If $G$ is a multigraph, then $\calS(G)$ is defined to be the set of all $n \times n$ symmetric matrices whereby the $ij$-entry is nonzero whenever $i \sim_! j$ (that is, there is exactly one edge between $i$ and $j$), the $ij$-entry is zero whenever $i \not\sim j$ and $i\ne j$, and diagonal entries may be any combination of zero or nonzero entries. (Note that entries corresponding to multiedges that are not a singleton edge may be zero or nonzero.)

For looped multigraphs, $\calS(G)$ is the set of matrices meeting both conditions above. However, we will not discern between loops and ``multiloops'', so for all practical purposes, all loops are simple.

In reverse, for a matrix $A$, the {\it underlying graph} of $A$ is the simple graph $G$ for which $A \in \calS(G)$.

Given a simple graph $G$, a \emph{loop configuration} is a looped graph whose underlying graph is $G$. We will let $\gloop$ be the loop configuration with all possible loops; and we will let $\gO$ be the looped configuration with no loops.

Similarly, an \emph{edge configuration} of a multigraph $G$ is a simple graph $H$ obtained from $G$ such that for each pair of vertices connected by multiedges, either one or no edge is kept. 

We will use the notation $\confl(G)$ $\confe(G)$ to denote the all of the loop and edge configurations of $G$ respectively. 

For simple graphs, $\calS(G)$ is the disjoint union of $\calS(H)$ over all the loop configurations $H$ of $G$. Notably, $\calS(G) \supsetneq \calS(\gO)$; hence, going forward, we must be careful to specify whether $G$ is a simple graph or a looped graph. 

Similarly, for a multigraph $G$, $\calS(G)$ is the disjoint union of $\calS(H)$ over all edge configurations $H$ of $G$.

Given a graph, looped graph or multigraph, $G$, one may wish to understand the possible spectra (eigenvalues) of matrices in $\calS(G)$. This is a challenging task to say the least. However, a simpler problem is to determine the maximum nullity.  For a general graph $G$, we define the \emph{maximum} nullity as 
\[M(G) = \max_{\bA \in \calS(G)} \nul(\bA).\]

A slightly more challenging problem that we will focus on is to determine the possible multiplicities of eigenvalues given a prescribed order. For a real symmetric matrix $A$,  we say that $A$ has {\it ordered multiplicity list} $(m_1, m_2, \ldots, m_\ell)$ if $A$ has with distinct (necessarily real) eigenvalues $\lambda_1 < \lambda_ 2 <  \cdots < \lambda_\ell$ with corresponding multiplicities $m_1, m_2, \ldots, m_\ell$. Observe that for a general graph $G$ and any $A \in \calS(G)$, it must be the case that any ordered multiplicity list  $(m_1, m_2, \ldots, m_\ell)$ has $m_i \le M(G)$.

\subsection{Zero Forcing on Graphs, Multigraphs, and Looped Graphs}

The classical {\it zero forcing} process is a one-player game played on a simple graph $G$. The player selects some set of vertices $S$ to initially colored blue; all others are uncolored. After which, the color change rule is iteratively applied: If a blue vertex has exactly one uncolored neighbor, it ``{\it forces}'' (or colors) that neighbor to become blue. The rule is applied until no more forces can be made. A set $S$ is a {\it zero forcing set} if after iteratively applying the rule all the vertices of $G$ will eventually be colored blue.  The goal of the game is to find the smallest zero forcing set, the size of which is called the \emph{zero forcing number} of the graph $G$, denoted $Z(G)$.

For looped graphs, the game is similar; however, the color change rule slightly different. If $i$ has a loop, then if all but one vertex in the closed neighborhood of $i$ (including $i$ itself) is blue, then the neighborhood forces that vertex to be blue; the distinction from before is that $i$ can be forced by its own neighborhood. And if $i$ does not have a loop, then whenever all but one vertex in the open neighborhood of $i$ (i.e., excluding $i$) is colored, $i$ can force that vertex to be blue; the distinction from the color change rule for simple graphs is that if $i$ does not have a loop, then it can force without being colored.

In effect, the classical zero forcing rule for simple graphs only allows a force if a force would be possible over any possible loop configuration.

For multigraphs, we take the color change rule to be where $i$ can only force $j$ whenever $i\sim_! j$ (i.e., $i$ and $j$ have a singleton edge). 

\begin{remark}
Let $G$ be a simple graph and $H$ a loop configuration of $G$.  Then $Z(H)\leq Z(G)$ by definition.
\end{remark}

For simple graph $G$, we define the \emph{enhanced zero forcing number}  \[\Zhat(G) = \max_{H\in \confl (G)} Z(H).\]

\begin{remark}
Let $G$ be a multigraph and $H$ an edge configuration of $G$.  Then $Z(H)\leq Z(G)$ by definition.
\end{remark}

For a multigraph graph $G$, we have \[\Zcheck(G) = \max_{H\in \confe (G)} Z(H).\]

\begin{theorem}[Barioli et al.\ \cite{treewidth}]
\label{aimgroup} 
For a simple graph $G$, \[ M(G) \le \Zhat(G) \le Z(G). \] 
\end{theorem}

\begin{theorem}[see \cite{cancun}] 
For a  multigraph $G$, \[ M(G)  \le \Zcheck(G)  \le Z(G). \] 
\end{theorem}

Hogben showed $M(G)\leq Z(G)$ for multigraphs $G$  \cite{cancun}, but the fact that $M(G)\leq \Zcheck(G)\leq Z(G)$ follows immediately from the definition of $\Zcheck(G)$.

\subsection{Positive Semi-definite Forcing and other variants}

The variations of zero forcing mentioned previously focused on the type of graph. In contrast, there are zero forcing variants that are motivated by further restrictions on matrices in $\mathcal{S}(G).$ Let $\mathcal{S}_+(G)$ be the set of all (symmetric) positive semi-definite matrices within $\mathcal{S}(G)$.  The \emph{positive semidefinite maximum nullity} of $G$ is 
\[M_+(G) = \max_{\bA \in \mathcal{S}_+(G)} \nul(\bA).\]

 Barrioli et al.\ defined {\it positive semidefinite forcing} for simple graphs \cite{zf_psd}  which was extended by Ekstrand et al.\ for multigraphs \cite{zf_psd_multi}. This variant is the same as zero forcing for the different types of general graphs, except with a subtle change to the color change rule. Let $X$ be the set of colored vertices, then consider the induced subgraph on $V(G) - X$ with components $Y_1, \ldots, Y_\ell$. A vertex $u$ can force an uncolored vertex $v$ if it could do so within any of the induced subgraphs on $X \cup Y_1, X \cup Y_2$ $\ldots$ or $X \cup Y_\ell$. In other words, for $u$ to force a vertex $v$, $v$ needs only to be the only relevant uncolored neighbor of $u$ among the same uncolored component as $v$. For multigraphs, forces can only occur on singleton edges however multiedges are still considered for determining the connected components.

\begin{theorem}[Bariloli, et al.\ \cite{zf_psd} and Ekstrand, et al.\ \cite{zf_psd_multi}] 
For simple graphs and multigraphs $G$, 

\[ M_+(G) \le Z_+(G) \]
\end{theorem}


Similarly, we may define 
\[\Zpcheck(G) = \max_{H\in \confe (G)} Z_+(H).\]
Thus, $M_+(G) \leq \Zpcheck(G) \leq Z_+(G)$.

It it worth remarking if $i \sim_! j$ in $G$, then for any matrix $\bA \in \mathcal{S}_+(G)$, it must be the case that $\bA_{ii}, \bA_{jj} \ne 0 $, as otherwise there is a  $2 \times 2$ principle submatrix with negative determinant. As a result, in the context of both for simple connected graph and positive semidefinite matrices and positive semidefinite forcing,  we can assume that every non-isolated vertex has loops.

Later in Section \ref{sec:skew}, we will consider {\it skew-forcing} where the matrices within $\calS(G)$ are restricted to skew-symmetric matrices. We will define the specific variations at that time.

\begin{table}
\begin{center}
    \begin{tabular}{| m{2.4cm} | m{1.5cm}| m{3.8cm} | m{3cm} |}
    \hline
    Name & Notation & Minimum Rank Problem & Color Change Rule  \\ \hline \hline
    
    Classical \cite{zf_aim} & $Z(G)$ & symmetric &  \\ \hline
    
    PSD forcing \cite{zf_psd} & $Z_+(G)$ & symmetric positive semi-definite  & forcing considers individual uncolored components \\ \hline
   
    Skew forcing \cite{zf_psd} & $Z_-(G)$ & skew-symmetric (or symmetric with 0 diagonal) & a vertex may force without being colored \\ \hline
    
    \end{tabular}
\end{center}
\caption{Summary of the different applications with their zero forcing variations.}
\label{tbl:summary}
\end{table}

\begin{proposition}
Let $G$ be a multigraph and $G'$ its (simple) underlying graph.  Suppose there is a minimum zero forcing set of $G'$, $S$ such that there are a sequence resulting forces to color all the vertices of $G$ that use only singleton edges in $G$.  Then, $S$ is also a zero forcing set of $G$, and $Z(G') = Z(G)$.
\end{proposition}

\begin{proof}
Suppose there is a zero forcing process on $G'$ starting with a minimum zero forcing set $S$ and using only singleton edges for each force.  Then every force in this process is also a valid force in $G$, so $S$ is also a zero forcing set of $G$.  Since $Z(G)\leq Z(G')$ by definition, $S$ is a minimum zero forcing set for $G$ and $Z(G) = Z(G')$.
\end{proof}

\section{Power zero forcing} \label{sec:power}

One of our main approaches will be to consider {\it powers of graphs}.  

For a simple graph a {\it lazy walk} is a walk that may remain at a vertex at each step, and its length is the number of steps. Let $G$ be a simple graph and $r$ a positive integer.  We define the multigraph $\Gamma(G,r)$ on the vertex set $V(G)$ such that the number of edges between $i$ and $j$ is the number of lazy walks from $i$ to $j$ of length at most $r$. 

In some sense, $\Gamma(G,r)$ is a graph power of $G$ as a multigraph. However, we will not generally be concerned with the exact number of edges between two vertices. Rather, for each pair of vertices $i,j$, we only truly consider whether there is an edge between $i$ and $j$, whether $i$ and $j$ form a singleton edge, or whether they form a multiedge.

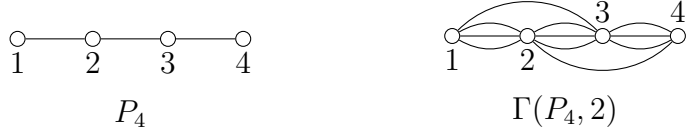
\begin{figure}
\begin{center}
\begin{tikzpicture}
\node[label={below:$1$}] (0) at (0,0) {};
\node[label={below:$2$}] (1) at (1,0) {};
\node[label={below:$3$}] (2) at (2,0) {};
\node[label={below:$4$}] (3) at (3,0) {};
\draw (0) -- (1) -- (2) -- (3);
\node[rectangle,draw=none] at (1.5,-1) {$P_4$};
\end{tikzpicture}
\hfil
\begin{tikzpicture}
\node[label={below:$1$}] (0) at (0,0) {};
\node[label={below:$2$}] (1) at (1,0) {};
\node[label={above:$3$}] (2) at (2,0) {};
\node[label={above:$4$}] (3) at (3,0) {};
\draw (0) -- (1);
\draw (0) to[bend left] (1);
\draw (0) to[bend right] (1);
\draw (1) -- (2);
\draw (1) to[bend left] (2);
\draw (1) to[bend right] (2);
\draw (2) -- (3);
\draw (2) to[bend left] (3);
\draw (2) to[bend right] (3);
\draw (0) to[bend left=45] (2);
\draw (1) to[bend right=45] (3);
\node[rectangle,draw=none] at (1.5,-1) {$\Gamma(P_4,2)$};
\end{tikzpicture}
\end{center}
\caption{An example of $G$ and $\Gamma(G,r)$, where $G=P_3$ and $r=2$.}
\label{fig:mulill}
\end{figure}

\begin{example}
Let $G = P_4$ be as shown in Figure~\ref{fig:mulill} and $r=2$.  Then there are three lazy walks of length at most $r$ from $1$ to $2$, namely, $(1,2)$, $(1,1,2)$, and $(1,2,2)$.  In contrast, there is only one lazy walk of length at most $r$ from $1$ to $3$, which is $(1,2,3)$, and there is no such lazy walk from $1$ to $4$.  Therefore, the graph $\Gamma(P_4,2)$ is as shown in Figure~\ref{fig:mulill}.  And the the edge configurations of $\Gamma(P_4,2)$ are the graph on the vertex set $\{1,2,3,4\}$ such that $\{1,3\}$ and $\{2,4\}$ are edges yet $\{1,4\}$ is not an edge.
\end{example}

\begin{remark}
Let $G$ be a graph and $\bA\in\calS(G)$.  Suppose $p(x)$ is a polynomial of degree $r$ with $r\geq 1$.  Then $p(\bA)$ is a matrix of $\calS(H)$ for some edge configuration  $H$ of $\Gamma(G,r)$. 
\end{remark}
     
\begin{definition}
Let $G$ be a simple graph.  Define \[\zr{r}(G) = \Zcheck(\Gamma(G,r)) \text{ and }
\zpr{r}(G) = \Zpcheck(\Gamma(G,r)).\]
\end{definition}

Note that $\zr{1}(G) = Z(G)$ and $\zpr{1}(G) = Z_+(G)$.

\begin{theorem}
\label{thm:main}
Let $G$ be a simple graph and $\bA\in\calS(G)$.  Suppose $m_1\geq \cdots \geq m_\ell$ are the eigenvalue multiplicities of $\bA$.  Then, for any $r=1,\ldots,\ell$,  
\[\sum_{i=1}^r m_i \leq \zr{r}(G).\]
\end{theorem}
\begin{proof}
Suppose $\bA$ has $q$ distinct eigenvalues $\lambda_1,\ldots,\lambda_q$ with multiplicities $m_1\geq \cdots \geq m_q$.  For a given $r=1,\ldots, n$, let $p(x)=(x-\lambda_1)\cdots(x-\lambda_r)$.  Thus, $p(\bA)$ is in $\calS(H)$ for some configuration $H$ of $\Gamma(G,r)$ and has nullity $\sum_{i=1}^rm_i$.  Therefore, 
\[\sum_{i=1}^r m_i = \nul(p(\bA)) \leq Z(H)\leq \zr{r}(G).\]
This completes the proof.
\end{proof}

\begin{example}
\label{ex:path}
Consider the path $P_n$ on vertices $\{v_1,\ldots,v_n\}$ in the path order.  Since $M(P_n)=1$, any matrix $A\in\calS(P_n)$ has $n$ distinct eigenvalue with multiplicities $m_1=\cdots =m_n=1$.  On the other hand, we claim that for any given $r\leq n$ and any edge configuration $H$ of $\Gamma(P_n,r)$, the set $B = \{1,\ldots,r\}$ is a zero forcing set of $H$.  To see this, first observe that $\{v_i,v_{i+r}\}$ is an edge in $H$ and $\{v_i,v_j\}$, $j\geq i+r+1$, is not an edge in $H$ regardless the choice of the edge configuration of $\Gamma(P_n,r)$.  Therefore, one may perform the forces $v_i\rightarrow v_{i+r}$ for $i=1,2,\ldots, n-r$ sequentially to color every vertex in the graph.  Therefore, 
\[r = \sum_{i=1}^r m_i \leq \zr{r}(P_n) \leq r\]
and the inequalities in Theorem~\ref{thm:main} are tight for any $r$.
\end{example}



Suppose a matrix $\bA\in\calS(G)$ has eigenvalue multiplicities $m_1\geq \cdots \geq m_\ell$.  Now we have two upper bounds for $\sum_{i=1}^r m_i$.  One is the upper bound $\zr{r}(G)$ given by Theorem~\ref{thm:main}, and the other is the upper bound $rZ(G)$ given by the classical zero forcing number.  The following two examples shows that they are, in general, not comparable.

\begin{example}
\label{ex:star}
Let $K_{1,n-1}$ be the star on vertices $\{v_1,\ldots,v_n\}$ such that $v_n$ is the center.  The adjacency matrix of $K_{1,n-1}$ has multiplicities $m_1=n-2$ and $m_2=m_3=1$.  Since any configuration of $\Gamma(K_{1,n-1},2)$ contains at least an edge (e.g., $\{v_1,v_2\}$), $\zr{2}(G)\leq n-1$.  Thus, $m_1+m_2=n-1$ implies $\zr{2}(G) = n-1$.  In this case, the bound $\zr{2}(G)$ outperforms the bound $2Z(G)$.
\end{example}

\begin{figure}
\begin{center}
\begin{tikzpicture}
\node [label={below:$1$}] (1) at (0,0) {};
\node [label={below:$2$}] (2) at (1,0) {};
\node [label={below:$3$}] (3) at (2,-0.5) {};
\node [label={below:$4$}] (4) at (3,0) {};
\node [label={below:$5$}] (5) at (4,0) {};
\node [label={above:$6$}] (6) at (1,1) {};
\node [label={above:$7$}] (7) at (2,0.5) {};
\node [label={above:$8$}] (8) at (3,1) {};

\draw (1) -- (2) -- (3) -- (4) -- (5);
\draw (6) -- (7) -- (8);
\draw (2) -- (7) -- (4);
\draw (3) -- (7);

\node[rectangle,draw=none] at (2,-1.5) {$G$};
\end{tikzpicture}
\hfil
\begin{tikzpicture}
\node [label={below:$1$}] (1) at (0,0) {};
\node [label={[label distance=0.2cm]below:$2$}] (2) at (1,0) {};
\node [label={below:$3$}] (3) at (2,-0.5) {};
\node [label={[label distance=0.2cm]below:$4$}] (4) at (3,0) {};
\node [label={below:$5$}] (5) at (4,0) {};
\node [label={above:$6$}] (6) at (1.3,1.3) {};
\node [label={[label distance=0.2cm]above:$7$}] (7) at (2,0.5) {};
\node [label={above:$8$}] (8) at (2.7,1.3) {};

\draw[line width=2pt] (1) -- (2) -- (3) -- (4) -- (5);
\draw[line width=2pt] (6) -- (7) -- (8);
\draw[line width=2pt] (2) -- (7) -- (4);
\draw[line width=2pt] (3) -- (7);
\draw[line width=2pt] (2) -- (4);
\draw (1) -- (3) -- (5) -- (7) -- (1);
\draw (6) -- (2) -- (8) -- (4) -- (6) -- (3) -- (8) -- (6);

\node[rectangle,draw=none] at (2,-1.5) {$\Gamma(G,2)$};
\end{tikzpicture}
\end{center}
\caption{A graph $G$ and the corresponding $\Gamma(G,2)$, where a thick line means two or more multi-edges.}
\label{fig:diffdrop}
\end{figure}
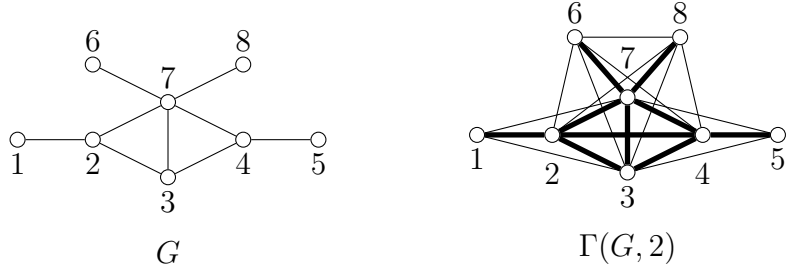

\begin{example}
Let $G$ be the graph shown in Figure~\ref{fig:diffdrop}.  Since $\{1,6\}$ is a zero forcing set of $G$, $Z(G)=2$.  On the other hand, one may calculate $\zr{2}(G) = 5$ since $Z(\Gamma(G,2))=5$, $Z(H)=5$, and 
\[Z(H)\leq \zr{2}(G)\leq Z(\Gamma(G,2)),\]
where $H$ is the configuration of $\Gamma(G,2)$ with all potential edges present.  Alternatively, the code \cite{mr_methematica} for computing $\zr{2}(G)$ is available.  Therefore, $2Z(G)\leq \zr{2}(G)$.
\end{example}

In general, Corollary~\ref{cor:partition} shows several upper bounds are available for the sum of multiplicities.

\begin{corollary}
\label{cor:partition}
Let $r\in [n]$ and ${r_1,\ldots, r_k}$ an integer partition of $r$.  Then 
\[\sum_{i=1}^r m_i\leq \sum_{i=1}^k \zr{r_i}(G).\]
\end{corollary}
\hfill $\qed$

Up until now, we have focused solely on the arbitrary sum of multiplicities with no consideration of the order of the eigenvalues.
Let $\Lambda=\{\lambda_1,\ldots,\lambda_q\}$ be a set of distinct real numbers with $\lambda_1<\cdots<\lambda_q$.  Any subset $S$ of $\Lambda$ can be partitioned into maximal consecutive segments; that is, $S=\bigcup_{i}s_i$ such that each $s_i$ is of the form $\{\lambda_a, \lambda_{a+1}, \ldots, \lambda_b\}$ for some $a$ and $b$.  If a segment contains $\lambda_1$ or $\lambda_n$, then it is called a \emph{boundary} segment.  We define the \emph{evenly consecutive order} of a segment $s_i$ as
\[e_\Lambda(s_i) = \begin{cases}
|s_i| & \text{if }s_i\text{ is boundary}; \\
2\left\lceil \frac{|s_i|}{2}\right\rceil & \text{otherwise},
\end{cases}\]
and the \emph{evenly consecutive order} of $S$ as 
\[e_\Lambda(S) = \sum_{i=1}^q e_\Lambda(s_i).\]
Note that the formula $2\left\lceil \frac{k}{2}\right\rceil$ is simply the smallest even number greater than or equal to $k$.

\begin{example}
\label{ex:poly}
If $\Lambda = \{\lambda_1,\ldots,\lambda_{10}\}$ is a set of real numbers with $\lambda_1<\cdots<\lambda_{10}$ and $S=\{\lambda_1,\lambda_5,\lambda_6,\lambda_7\}$, then the maximal consecutive segments of $S$ are $s_1=\{\lambda_1\}$ and $s_2=\{\lambda_5,\lambda_6,\lambda_7\}$, where $s_1$ is boundary and $s_2$ is not.  Thus, we have $e_\Lambda(s_1) = 1$, $e_\Lambda(s_2) = 4$, and $e_\Lambda(S) = 5$.  Under this setting, one may construct a polynomial 
\[p(x) = (x-\lambda_1)(x-\lambda_5)(x-\lambda_6)(x-\lambda_7)^2\]
of degree $e_\Lambda(S) = 5$ such that $p(\lambda) = 0$ if $\lambda\in S$ and $p(\lambda)>0$ if $\lambda\in\Lambda\setminus S$.  The polynomial is shown in in Figure~\ref{fig:poly}.
\end{example}

\begin{figure}[h]
\begin{center}
\begin{tikzpicture}[scale=0.6]
\draw[-triangle 45] (-0,0) -- (11,0);
\foreach \i in {1,...,10} {
\pgfmathsetmacro{\x}{\i}
\draw (\x,-0.1) -- (\x,0.1);
\coordinate (\i) at (\x,0);
\node[rectangle,draw=none,below] at (\x,-0.2) {$\lambda_{\i}$};
}
\draw plot [samples=200,domain=0.5:5] function {-0.25*(x-1)*(x-5)};
\draw plot [samples=200,domain=5:6] function {1*(x-5)*(x-6)};
\draw plot [samples=200,domain=6:7.5] function {2*(x-6)*(x-7)*(x-7)};
\end{tikzpicture}    
\end{center}
\caption{An illustration of $p(x)$ in Example~\ref{ex:poly}}
\label{fig:poly}
\end{figure}
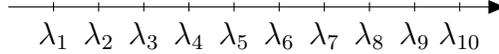

\begin{theorem}
\label{thm:mainzp}
Let $G$ be a graph and $\bA\in\calS(G)$ with distinct eigenvalues $\Lambda=\{\lambda_1,\ldots,\lambda_q\}$, $\lambda_1 < \cdots < \lambda_q$, and the corresponding multiplicities $m_1,\ldots, m_q$.  Then 
\[\sum_{\lambda_i\in S}m_i \leq \zpr{r}(G)\]
for any $S\subseteq\Lambda$ with $r = e_\Lambda(S)$.
\end{theorem}

\begin{proof}
Let $\bA$ be a matrix in $\calS(G)$ with eigenvalues $\lambda_1<\cdots<\lambda_q$ and multiplicities $m_1,\ldots,m_q$, respectively.  For a given $r=1,\ldots,n$ and a subset $S$ of $\Lambda$ with $e_\Lambda(S) = r$, there is a polynomial $p(x)$ such that $p(\lambda)=0$ if $\lambda\in S$ and $p(\lambda)> 0$ if $\lambda\in\Lambda\setminus S$.  As a consequence, $p(\bA)$ is a positive semidefinite matrix in $\calS(\Gamma(G,r))$ and has nullity $\sum_{\lambda_i\in S} m_i$.  This means 
\[\sum_{\lambda_i\in S} m_i = \nul(p(\bA)) \leq Z_+(H) \leq \zpr{r}(G).\]
This completes the proof.
\end{proof}

\begin{example}
\label{ex:path+}
For similar reasons in Examples~\ref{ex:path}, $\zpr{r}(P_n) = r$ for any $r\leq n$.
\end{example}

Recall that $q(G)$ is the minimum number of distinct eigenvalues among $\bA\in\calS(G)$.

\begin{corollary}
\label{cor:qlbd}
Let $G$ be a graph on $n$ vertices.  Let $\widehat{r}$ be the smallest $r$ such that $\zr{r}(G)=n$.  Then $q(G)\geq\widehat{r}$.
\end{corollary}
\hfill $\qed$

On a simple graph $G$, a path from vertex $i$ to vertex $j$ is called a \emph{unique shortest path}.  Let $\widehat{k}$ be the number of vertices on the longest unique shortest path between arbitrary two vertices of $G$.  It was known \cite{AACFMN13} that $q(G)\geq\widehat{k}$.  The next proposition explains that the $\widehat{r}$ in Corollary~\ref{cor:qlbd} is the same as $\widehat{k}$.

\begin{proposition}
Let $G$ be a graph on $n$ vertices and $r>0$ an integer.  Then the following are equivalent.
\begin{enumerate}[label={\rm(\arabic*)}]
    \item $\zr{r}(G) = n$.
    \item $\zpr{r}(G) = n$.
    \item The edge configurations of\/ $\Gamma(G,r)$ contain the empty graph.
    \item Any unique shortest path on $G$ contains at most $r$ vertices.
\end{enumerate}
\end{proposition}
\begin{proof}
Note that $Z(H) = |V(H)|$ if and only if $Z_+(H) = |V(H)|$, and these two conditions are equivalent to $H$ is an empty graph.  Therefore, (1), (2), and (3) are equivalent.  

Suppose the longest unique shortest path on $G$ is on $\ell$ vertices, namely, $v_1,\ldots,v_\ell$.  For any $k\leq \ell-1$, the number of lazy walks from $v_1$ to $v_{k+1}$ of length at most $k$ is $1$, so there is a singleton edge between $v_1$ and $v_{k+1}$ in $\Gamma(G,k)$.  If $k\geq \ell$, then between any $i,j\in V(G)$, there are at least two lazy walks of length at most $k$, so the empty graph is a configuration of $\Gamma(G,k)$.  Conversely, if $\Gamma(G,r)$ contains the empty graph, then by definition, there is no unique shortest path of length $r$ or less. 
\end{proof}

\begin{theorem}
\label{thm:tree}
Let $T$ be a tree. Let $L_1$ be set of leaves of $T$ and $\ell_1=|L_1|$. Let $L_j$ be the set of vertices whose shortest distance to a leaf is $j-1$ and denote $\ell_i=|L_i|$.  Then,
\[\zr{k}(T)\leq \sum_{i=1}^k \ell_i - 1\]
for which $L_k$ is nonempty.
\end{theorem}

\begin{proof}
We show that $\zr{k}(T)\leq \sum_{i=1}^k \ell_i - 1$ by carefully choosing a path $P$ with $k$ vertices and showing $ \left( \bigcup_{i=1}^k L_i \right)  \setminus V(P)$  is a zero forcing set for any edge configuration of $\Gamma(T,k)$.

Pick a vertex $v_k$ in $L_k$.  By the definition of $L_k$, there is a path $P$ between $v_k$ and a leaf $v_1$ of distance $k-1$.  Label $V(P)$ by $v_k, \ldots, v_1$ following the path order.

Now, pick an arbitrary edge configuration $H$ of $\Gamma(G,k)$ and color every vertices in $\left ( \bigcup_{i=1}^k L_i \right) \setminus V(P)$ blue.

\medskip

{Claim:} All components of $T - P$ can be colored in $H$.

\medskip

{\it Proof of Claim:}
Choose a component of $T-P$, and call it $T'$.  Inductively, for any $j\geq k$, if vertices in $\bigcup_{i=1}^j L_i \cap T'$  are blue, then $L_{j+1} \cap T'$ may turn blue in the next step.  Let $x\in L_{j+1} \cap T'$ be a white vertex.  Then $x$ has a neighbor $y\in L_{j+1-k} \cap T'$ in $H$.  By the inductive hypothesis, $y$ is blue and every neighbor of $y$ is blue except for $x$, so $y$ may force $x$ to blue.  Therefore, $L_{k+1} \cap T',L_{k+2} \cap T',\ldots$ will be blue and eventually every vertex in $T'$ is blue. The completes the proof of the claim. \hfill $\triangle$

\medskip

Since all vertices except for $v_1, \ldots, v_k$ can necessarily be colored, it remains to show that $P$ itself can be colored. 

For $v_k$, there must be a vertex, $u_k$, such that $d(v_k,u_k)= k$ and $d(v_{k-1},u_k) = k+1$ and $u_k$ is colored (if not, then $v_k$ is less than distance $k$ away from a leaf, and hence not in $L_k$.) Therefore, all vertices of within distance $k$ of $u_k$ are colored, except for $v_k$ and $u_k$ forces $v_k$. From there, consider the path, $u_k, u_{k-1}, u_{k-2}, \ldots u_{1}, v_{k}$ in $T$. Necessarily, since $T$ is a tree, $u_i$ is exactly distance $k$ from $v_i$ and all other vertices within distance $k$ of $u_i$ are colored (otherwise, $u_k$ would not be distance k from $v_k$ in $T$). Hence, inductively, we have that, starting with $i=k$ as above and decrementing $i$, $u_i$ forces $v_i$. This completes the proof.
\end{proof}

We remark that for $i=1$, Theorem \ref{thm:tree} says that the maximum nullity of a tree is at most the number of leaves minus one. In contrast, it is known that the maximum nullity is exactly the path-cover number of the tree. In which case, Theorem \ref{thm:tree} can be off by a factor of 2. Hence, for most trees, it is also likely that for higher values of $k$, Theorem \ref{thm:tree} does not achieve equality for most trees.

\begin{proposition}
Let $G$ be a graph on $n$ vertices.  The following are equivalent.
\begin{enumerate}[label={\rm(\arabic*)}]
\item $\zr{r}(G) = r$ for some $r$.
\item $\zpr{r}(G) = r$ for some $r$.
\item $q(G) = n$
\item $M(G) = 1$
\item $G$ is a path.
\end{enumerate}
\end{proposition}

\begin{proof}
By definition, (3) and (4) are equivalent.  It is known that (4) and (5) are equivalent; see, e.g., \cite{gSAP}.  If $G$ is a path, then (1) and (2) are true by Examples~\ref{ex:path} and \ref{ex:path+}.  

Suppose $\zpr{r}(G) = r$ for some r.  Let $m_1, \ldots, m_r$ be any $r$ eigenvalue multiplicities of a matrix $A\in\calS(G)$.  Since eigenvalue multiplicities are at least one, 
\[r\leq \sum_{i=1}^r m_i \leq \zpr{r}(G)\leq \zr{r}(G) = r\]
and thus $\zpr{r}(G) = r$.  Moreover, since the choices of $m_1,\ldots,m_r$ and $A\in\calS(G)$ are arbitrary, $M(G)=1$.
\end{proof}

\begin{theorem}
Let $G$ be a connected graph on $n$ vertices.  Then the following are equivalent.
\begin{enumerate}[label={\rm(\arabic*)}]
\item $\zr{2}(G)\leq 3$.
\item $q(G) \geq n-1$.
\item $G$ is either a path with an extra edge joining two vertices of distance two, a path with a leaf on an internal vertex, or a path.
\end{enumerate}
\end{theorem}
\begin{proof}
According to \cite[Theorem~51]{gSAP}, (2) and (3) are equivalent.  Suppose $\zpr{2}(G) \leq 3$.  Let $\bA\in\calS(G)$ with eigenvalue multiplicities $m_1\geq m_2\geq\cdots\geq m_q$.  Since $m_1+m_2\leq 3$ and $m_2\geq 1$, we have $m_1\leq 2$ and $1\geq m_2\geq \cdots \geq m_q$.  Therefore, $q(G)\geq n-1$.  

Let $G$ be a path with an extra edge joining two vertices of distance two.  Thus, $G$ can also be obtained from a $P_{n-1}$, labeled by $v_1,\ldots,v_{n-1}$, by adding a new vertex $x$ joining two consecutive vertices.  Under this setting, the $\{v_1,v_2,x\}$ is a zero forcing set for any edge configuration of $\Gamma(G,2)$.  Hence $\zr{2}(G)\leq 3$.

Let $G$ be a path with a leaf $x$ on an internal vertex.  Similarly, $\{v_1,v_2,x\}$ is a zero forcing set for any edge configuration of $\Gamma(G,2)$.  Hence $\zr{2}(G)\leq 3$.

Finally, Example~\ref{ex:path} has that $\zr{2}(P_n)=2\leq 3$.
\end{proof}

\begin{theorem}
Let $G$ be a connected graph on $n$ vertices.  Then the following are equivalent.
\begin{enumerate}[label={\rm(\arabic*)}]
\item $\zpr{2}(G)\leq 3$.
\item $M(G)\leq 2$ and any matrix in $\calS(G)$ does not have consecutive multiple eigenvalue.
\item $G$ is either a generalized $3$-star, a generalized $3$-sun, a path with an extra edge joining two vertices of distance two, or a path.
\end{enumerate}
\end{theorem}
\begin{proof}
Suppose $\zpr{2}(G)\leq 3$.  Then $m_i+m_j\leq 3$ for any two consecutive eigenvalues, so $M(G)\leq 2$ and there is no consecutive multiple eigenvalues.  

Suppose $G$ is a graph not allowing two consecutive multiple eigenvalues.  Then by \cite[Corollary~5.5]{minormult}, $G$ is either a generalized star, a generalized $3$-sun, a path with an extra edge joining two vertices of distance $2$, or a path.  By examining the maximum nullities of these graphs, (2) implies (3).

If $G$ is a generalized $3$-star, let $v$ be the center vertex, and $x,y$ any two of the three neighbors of $v$.  If $G$ is a generalized $3$-sun or a a path with an extra edge joining two vertices of distance two, let $v,x,y$ be the three vertices on the unique cycle.  Thus, $\{v,x,y\}$ is a PSD zero forcing set of any edge configuration of $\Gamma(G,2)$.  Along with Example~\ref{ex:path+}, (3) implies (1).
\end{proof}

\subsection{Restrictions caused by \texorpdfstring{$K_{2,3}$}{K23}}

Sometimes not all configurations in $\Gamma(G,k)$ is a graph of $A^k$ for some $A\in\calS(G)$.  Here we will see some examples where $K_{2,3}$ and $K_{2,3}+e$ limits the achievable configurations in $\Gamma(G,2)$ and provides a detailed description on the multiplicity lists.

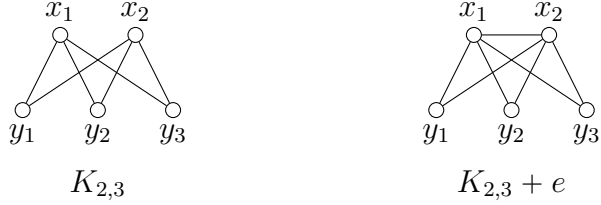
\begin{figure}[h]
    \centering
    \begin{tikzpicture}
    \node[label={below:$y_1$}] (y1) at (0,0) {};
    \node[label={below:$y_2$}] (y2) at (1,0) {};
    \node[label={below:$y_3$}] (y3) at (2,0) {};
    \node[label={above:$x_1$}] (x1) at (0.5,1) {};
    \node[label={above:$x_2$}] (x2) at (1.5,1) {};
    \foreach \i in {1,2,3} { 
        \foreach \j in {1,2} {
            \draw (y\i) -- (x\j);
        }
    }
    \node[rectangle,draw=none] at (1,-1) {$K_{2,3}$};
    \end{tikzpicture}
    \hfil
    \begin{tikzpicture}
    \node[label={below:$y_1$}] (y1) at (0,0) {};
    \node[label={below:$y_2$}] (y2) at (1,0) {};
    \node[label={below:$y_3$}] (y3) at (2,0) {};
    \node[label={above:$x_1$}] (x1) at (0.5,1) {};
    \node[label={above:$x_2$}] (x2) at (1.5,1) {};
    \foreach \i in {1,2,3} { 
        \foreach \j in {1,2} {
            \draw (y\i) -- (x\j);
        }
    }
    \draw (x1) -- (x2);
    \node[rectangle,draw=none] at (1,-1) {$K_{2,3}+e$};
    \end{tikzpicture}
    \caption{Labeled $K_{2,3}$ and $K_{2,3} + e$}
    \label{fig:k23e}
\end{figure}

\begin{lemma}
\label{lem:k23e}
Let $G$ be a graph with an induced $K_{2,3}$ or $K_{2,3}+e$ whose two parts are $X = \{x_1,x_2\}$ and $Y=\{y_1,y_2,y_3\}$.  
Suppose the only paths of length two from $y_i$ to $y_j$, $i\neq j$, are through $x_1$ and $x_2$.  
Then for any $A\in\calS(G)$, the graph of $A^2$ has at least one edge on $Y$, so the sum of any two eigenvalue multiplicities of $A$ is bounded above by 
\[\max_{\substack{H \in \confe(\Gamma(G,2))\\ E(H[Y])\neq\emptyset}} Z(H). \] 
\end{lemma}
\begin{proof}
Let $A = \begin{bmatrix}a_{uv}\end{bmatrix}$.
Consider the three products $a_{x_1y_1}a_{y_1x_2}$, $a_{x_1y_2}a_{y_2x_2}$, $a_{x_1y_3}a_{y_3x_2}$.  Since $A\in\calS(G)$, these products are nonzero.  By the pigeonhole principle, two of them have the same sign, say $(a_{x_1y_i}a_{y_ix_2})(a_{x_1y_j}a_{y_jx_2})>0$.
Equivalently, this means
$(a_{y_ix_1}a_{x_1y_j})(a_{y_ix_2}a_{x_1y_j})>0$.  Therefore, the $y_iy_j$-entry of $A^2$ is nonzero.

Let $\lambda_1$ and $\lambda_2$ be two eigenvalues of $A$ with multiplicities $m_1$ and $m_2$.  Then the matrix $(A-\lambda_1I)(A-\lambda_2I)$ has nullity $m_1+m_2$, and its $y_iy_j$-entry is nonzero, so its graph is some graph $H\in\confe(\Gamma(G,2))$ with $E(H[Y])\neq\emptyset$.
\end{proof}

\begin{example}
\label{ex:k23e}
Let $G$ be $K_{2,3}$ or $K_{2,3}+e$.  The configurations of $\Gamma(G,2)$ can be any graph on $5$ vertices, so $\zr{2}(G)=5$.  Meanwhile, the longest unique shortest path on $K_{2,3}$ are on $2$ vertices, so it seems that $q(G)$ can possibly be $2$.  However, if we only focus on the configurations $H$ of $\Gamma(G,2)$ with at least an edge, then its zero forcing number is at most $4$ since one may color every vertex except for one of the two endpoints of the edge.  By Lemma~\ref{lem:k23e} the sum of any two multiplicities is bounded above by $4$ and $q(G)\geq 3$.
\end{example}

\begin{figure}[h]
    \centering
    \begin{tikzpicture}
    \node[label={below:$y_1$}] (y1) at (0,0) {};
    \node[label={below:$y_2$}] (y2) at (1,0) {};
    \node[label={below:$y_3$}] (y3) at (2,0) {};
    \node[label={above:$x_1$}] (x1) at (0.5,1) {};
    \node[label={above:$x_2$}] (x2) at (1.5,1) {};
    \foreach \i in {1,2,3} { 
        \foreach \j in {1,2} {
            \draw (y\i) -- (x\j);
        }
    }
    \node[label={above:$z$}] (z) at (2.5,2) {};
    \draw (z) -- (x1);
    \draw (z) -- (x2);
    \draw (z) -- (y3);
    \node[rectangle,draw=none] at (1,-1) {\texttt{G170}};
    \end{tikzpicture}
    \hfil
    \begin{tikzpicture}
    \node[label={below:$y_1$}] (y1) at (0,0) {};
    \node[label={below:$y_2$}] (y2) at (1,0) {};
    \node[label={below:$y_3$}] (y3) at (2,0) {};
    \node[label={above:$x_1$}] (x1) at (0.5,1) {};
    \node[label={above:$x_2$}] (x2) at (1.5,1) {};
    \foreach \i in {1,2,3} { 
        \foreach \j in {1,2} {
            \draw (y\i) -- (x\j);
        }
    }
    \draw (x1) -- (x2);
    \node[label={above:$z$}] (z) at (2.5,2) {};
    \draw (z) -- (x1);
    \draw (z) -- (x2);
    \draw (z) -- (y3);
    \node[rectangle,draw=none] at (1,-1) {\texttt{G179}};
    \end{tikzpicture}
    \caption{Labeled \texttt{G170} and \texttt{G179}}
    \label{fig:g170g179}
\end{figure}

\begin{example}
Following the same arguments as in Example~\ref{ex:k23e}, the two graphs $G$ in Figure~\ref{fig:g170g179} have $q(G)\geq 3$ but they do not have any unique shortest path on $3$ vertices.
\end{example}

Let $G_\ell$ be obtained from $P_\ell$ and $K_{2,3}$ by joining the $y_1$ in $K_{2,3}$ with an endpoint of $P_\ell$.  Let $G_\ell+e$ be the graph obtained from $G_\ell$ by adding the edge $\{x_1,x_2\}$.  See Figure~\ref{fig:pathk23e}.

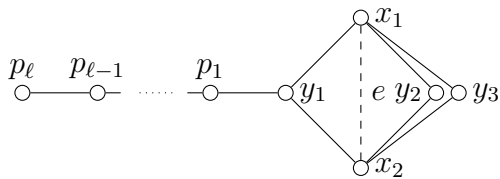
\begin{figure}[h]
    \centering
    \begin{tikzpicture}
    \node[label={right:$y_1$}] (y1) at (0,0) {};
    \node[label={left:$y_2$}] (y2) at (2,0) {};
    \node[label={right:$y_3$}] (y3) at (2.3,0) {};
    \node[label={right:$x_1$}] (x1) at (1,1) {};
    \node[label={right:$x_2$}] (x2) at (1,-1) {};
    \foreach \i in {1,2,3} {
        \foreach \j in {1,2} {
            \draw (y\i) -- (x\j);
        }
    }
    \draw[dashed] (x1) to node[midway,right,draw=none]{$e$} (x2);
    
    \node[label={above:$p_1$}] (p1) at (-1,0) {};
    \node[label={above:$p_{\ell-1}$}] (plm) at (-2.5,0) {};
    \node[label={above:$p_\ell$}] (pl) at (-3.5,0) {};

    \draw (y1) -- (p1);
    \draw (pl) -- (plm);
    \draw (p1.west) -- ++(-0.2,0);
    \draw (plm.east) -- ++(0.2,0);
    \draw[dotted] ([xshift=-0.4cm]p1.west) -- ([xshift=0.4cm]plm.east);
    \end{tikzpicture}
    \caption{The graph $G_\ell$ and the optional edge $e$}
    \label{fig:pathk23e}
\end{figure}

\begin{theorem}
\label{thm:pathk23e}
Let $G$ be $G_\ell$ or $G_\ell+e$ as shown in Figure~\ref{fig:pathk23e} and $A\in\calS(G)$.  Then the sum of any two eigenvalue multiplicities of $A$ is at most $4$.
\end{theorem}

\begin{proof}
Let $A\in\calS(G_\ell)$ and $\lambda_1,\lambda_2$ two eigenvalues of $A$ with multiplicities $m_1,m_2$.  Let  $H\in\confe(\Gamma(G,2))$ be the graph of $(A-\lambda_1I)(A-\lambda_2I)$.  Then $H$ contains no edges between vertices in $\{y_2,y_3\}$ and vertices in $\{p_1,\ldots,p_\ell\}$.

According to Lemma~\ref{lem:k23e}, $H$ contains at least one of $\{y_1,y_2\}$, $\{y_1,y_3\}$, and $\{y_2,y_3\}$ as an edge.
\begin{itemize}
    \item If $\{y_1,y_2\}\in E(H)$, let $u=y_2$ and $v=y_1$.  
    \item If $\{y_1,y_3\}\in E(H)$, let $u=y_3$ and $v=y_1$.  
    \item If $\{y_2,y_3\}\in E(H)$, let $u=y_3$ and $v=y_2$.   
\end{itemize}
Let $U = \{x_1,x_2,x_3,y_1,y_2\}$.  Now $U\setminus\{v\}$ is a zero forcing set of $H$ by the process $u\rightarrow v$, $x_1\rightarrow p_1$, $y_1\rightarrow p_2$, $p_1\rightarrow p_3$, $\ldots$, $p_{\ell-2}\rightarrow p_\ell$.  In either case, $Z(H)\leq 4$, so $m_1+m_2\leq 4$.
\end{proof}

\begin{remark}
The graphs $G_1$ and $G_1+e$ are \texttt{G125} and \texttt{G138} in \emph{An Atlas of Graphs}\cite{atlas}.  In a previous study, \cite[Section~4.3]{ahn2017ordered}, these two graphs are the ``remaining case'' that need additional efforts to rule out the ordered multiplicity lists 
(1,3,2) and (2,3,1).
\end{remark}

The next example shows the similar techniques in Lemma~\ref{lem:k23e} can be applied to other graphs.  

\begin{figure}[h]
    \centering
    \begin{tikzpicture}
    \node[label={below:$c$}] (c) at (0,0) {};
    \foreach \i in {1,2,3,4,5} {
        \pgfmathsetmacro{\ang}{90 + 72*(\i-1)}
        \node[label={\ang:$w_\i$}] (w\i) at (\ang:1) {};
        \draw (c) -- (w\i);
    }
    \draw (w1) -- (w2) -- (w3) -- (w4) -- (w5) -- (w1);
    \end{tikzpicture}
    \caption{A wheel graph $W_6$.}
    \label{fig:wheel6}
\end{figure}
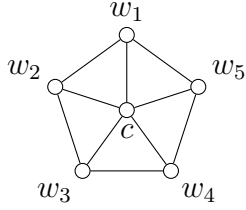

\begin{example} \label{ex:wheel}
Let $G=W_6$ be the wheel graph as shown in Figure~\ref{fig:wheel6} and $A=\begin{bmatrix}a_{uv}\end{bmatrix}\in\calS(G)$.  By replacing $A$ with $DAD$, where $D$ is a diagonal matrix whose diagonal entries are $1$ or $-1$, we may assume $a_{cw_i}>0$ for each $i=1,\ldots,5$.  

Let $H$ be the graph of $A^2$.  By writing $+$ or $-$ on the $5$-cycle induced on $\{w_1,\ldots,w_5\}$, there must be two consecutive edges $\{w_i,w_{i-1}\}$ and $\{w_i,w_{i+1}\}$ with the same signs, where the index are modulo $5$.  That is,  $a_{w_{i-1}w_i}a_{w_iw_{i+1}}>0$.  Since $a_{w_{i-1}c}a_{cw_{i+1}}>0$, the $(w_{i-1},w_{i+1})$-entry of $A^2$ is nonzero.  Therefore, $(A-\lambda_1I)(A-\lambda_2I)$ is not a zero matrix for any eigenvalues $\lambda_1,\lambda_2$ of $A$, so $q(G)\geq 3$.  Note that $W_6$ is the graph \texttt{G187} in the atlas \cite{atlas}, and this provides an alternative proof of \cite[Lemma~6.14]{BHPRT18}.
\end{example}

\subsection{A bound for \texorpdfstring{$q(G)$}{q(G)}}

 We now show that our techniques with some extra analysis are able to show that $q(\TT)>\diam(\TT)+1$ for the $\TT$ shown in Figure~\ref{fig.bftree}.  This example was provided by Barioli and Fallat~\cite[Fig 3.1]{barioli2004two}.

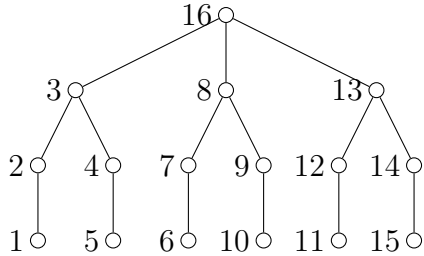
\begin{figure}[h]
    \centering
    \begin{tikzpicture}
    \node[label={left:$1$}] (1) at (0,0) {};
    \node[label={left:$5$}] (5) at (1,0) {};
    \node[label={left:$6$}] (6) at (2,0) {};
    \node[label={left:$10$}] (10) at (3,0) {};
    \node[label={left:$11$}] (11) at (4,0) {};
    \node[label={left:$15$}] (15) at (5,0) {};
    \node[label={left:$2$}] (2) at (0,1) {};
    \node[label={left:$4$}] (4) at (1,1) {};
    \node[label={left:$7$}] (7) at (2,1) {};
    \node[label={left:$9$}] (9) at (3,1) {};
    \node[label={left:$12$}] (12) at (4,1) {};
    \node[label={left:$14$}] (14) at (5,1) {};
    \node[label={left:$3$}] (3) at (0.5,2) {};
    \node[label={left:$8$}] (8) at (2.5,2) {};
    \node[label={left:$13$}] (13) at (4.5,2) {};
    \node[label={left:$16$}] (16) at (2.5,3) {};
    
    \draw (1) -- (2) -- (3) -- (4) -- (5);
    \draw (6) -- (7) -- (8) -- (9) -- (10);
    \draw (11) -- (12) -- (13) -- (14) -- (15);
    \draw (3) -- (16) -- (13);
    \draw (16) -- (8);
    \end{tikzpicture}
    \caption{A tree $\TT$ with $q(\TT)=8$ and $\diam(\TT)+1=7$}
    \label{fig.bftree}
\end{figure}

\begin{theorem}\label{thm:bftree} 
{\rm \cite{barioli2004two}}
Let $\TT$ be the graph in Figure~\ref{fig.bftree}.  Then $q(\TT) = 8$ while $\diam(\TT) + 1 = 7$.
\end{theorem}

Here, we provide an alternative proof of Theorem~\ref{thm:bftree} using just $Z$, $Z_+$, $Z^{(r)}$, and $Z^{(r)}$ (for $r= 2$ and $5$). This effectively abstracts all of the linear algebra in place of zero forcing parameters.

\begin{lemma} 
\label{lem.t321basic}
Let $\TT$ be the graph shown in Figure~\ref{fig.bftree}, we have
\begin{itemize}
\item $Z_+(\TT) = 1$,
\item $Z(\TT) = 4$, and
\item $Z^{(2)}(\TT) = 7$.
\end{itemize}
\label{lem:bfzs}
\end{lemma} 

\begin{proof}
The fact that $Z_+(\TT) = 1$ and $Z(\TT) = 4$ come from direct computation (or the algorithm in \cite{FH07}).  Using the labels in Figure \ref{fig.bftree}, $\{1, 2, 6, 7, 11, 12, 16\}$ is a zero forcing set for any edge configuration of $\Gamma(\TT,2)$, so $\zr{2}(\TT) \leq 7$.  (This direction of inequality is all we need for proving Theorem~\ref{thm:bftree}.)  The optimality of this zero forcing sets can be verified by exhaustion via computer, see \cite{mr_software}.
\end{proof}

\begin{lemma}
\label{lem:dtree}
For any edge configuration $H$ of $\Gamma(\TT,5)$, either $Z(H) \le 13$ or $Z_+(H) \le 11$.  
\end{lemma}

\begin{proof}
Let $H$ be a edge configuration of $\Gamma(\TT,5)$.  We will use the labels in Figure~\ref{fig.bftree}.

Suppose $\{1,5\}$, $\{6,10\}$, and $\{11,15\}$ are edges in $H$.  Then the subset $S_1 = V(H)\setminus\{5,10,15\}$ is a zero forcing set of $H$ since $1,6,11$ will force $5,10,15$ to be blue, respectively.  Therefore, $Z(H)\leq 13$ in this case.

In fact, $S_1$ is still a zero forcing set of $H$ even if one of $\{1,5\}$, $\{6,10\}$, and $\{11,15\}$ is not an edge, say $\{11,15\}$.  One may perform the forces $1\rightarrow 5$ and $6\rightarrow 10$, then $2\rightarrow 15$ to color every vertex.  Similarly, $Z(H)\leq 13$ in this case. 

Suppose at least two of $\{1,5\}$, $\{6,10\}$, and $\{11,15\}$ are not an edge in $H$, say $\{6,10\}$, and $\{11,15\}$.  (Here $\{1,5\}$ might or might not be an edge of $H$.)  By the assumption, $\{5,6,10,11,15\}$ is an independent set since two leaves from different branches, e.g., $5$ and $6$, are of distance $6$ and are not adjacent to each other in $H$.  Let $S_2 = V(H)\setminus\{5,6,10,11,15\}$.  Then $S_2$ is a PSD zero forcing set of $H$ since each of $\{5,6,10,11,15\}$ is adjacent to a blue vertex in $H$ (of distance $5$ in $\TT$) and this vertex can force it to be blue.  Therefore, $Z_+(H)\leq 11$ for the remaining cases.
\end{proof}

\begin{corollary}
\label{cor.t321}
Let $A\in\calS(\TT)$.  Then one of the following holds.
\begin{itemize}
    \item The sum of any five eigenvalue multiplicities of $A$ is at most $13$.
    \item The sum of any five consecutive eigenvalue multiplicities of $A$ is at most $11$.
\end{itemize}
\end{corollary}

We are now ready to prove Theorem~\ref{thm:bftree}

\begin{proof}[Proof of Theorem~\ref{thm:bftree}]
It is obvious that $\diam(\TT) + 1 = 7$, and it is known that $7$ is a lower bound for $q(\TT)$.  Suppose, for the purpose of yielding a contradiction, that $q(\TT) = 7$.

Let $A$ be a matrix in $\calS(\TT)$ with $7$ distinct eigenvalue 
\[\lambda_1 \leq \lambda_2 \leq \cdots \leq \lambda_7.\]
Let $m_i$ be the multiplicity of $\lambda_i$ for $i=1,\ldots, 7$.  Since $Z_+(\TT)=1$ Lemma~\ref{lem:bfzs}, $m_1=m_7=1$.  

According to Corollary~\ref{cor.t321}, one of the two cases must holds.  Since 
\[m_2 + \cdots + m_6 = 16 - m_1 - m_7 = 14,\]
the first case in Corollary~\ref{cor.t321} does not hold.  Therefore the sum of any five consecutive eigenvalue multiplicities of $A$ is at most $11$.  However, this means \[m_1 + \cdots + m_5 \leq 11 \implies m_6 \geq 4\]
and 
\[m_3 + \cdots + m_7 \leq 11 \implies m_2 \geq 4.\]
Consequently, $m_2 + m_6 \geq 8$, violating the fact $\zr{2}(\TT)\leq 7$ by Lemma~\ref{lem.t321basic}.  Therefore, $q(\TT) = 8$.  (The adjacency matrix of $\TT$ with the diagonal entry of $16$ set to $1$ has $8$ distinct eigenvalues.)

\end{proof}

\begin{remark}
Kim and Shader~\cite{kim2009smith} generalized $\TT$ into the family of $(k,\ell)$-whirl graphs, where $k$ is the degree of the center vertices and $\ell$ is the number of vertices of the pending paths starting from the third level.  Thus, $\TT$ is the $(3,2)$-whirl.  Similar arguments in this subsection show that $q(W')>\diam(W) + 1$ for all $(3,\ell)$-whirl.  However, Kim and Shader~\cite{kim2009smith} showed using more technical methods that $q(W')\geq \frac{9}{8}\diam(W') + \frac{1}{2}$ for all $(3,\ell)$-whirl.
\end{remark}

We also have a more general bound for $q(G)$:

\begin{theorem}\label{thm:qbound}
Let $G$ be a connected graph on $n$ vertices. Choose a positive even integer $k \le \diam(G)$. Then, the minimum number of distinct eigenvalues for any matrix $\bA \in \mathcal{S}(G)$, $q(G)$, obeys

\[ q(G) \ge \frac{k n - k^2 Z(G)} {Z^{(k)}_+(G)}. \] 
\end{theorem}

\begin{proof}
Choose  $\bA \in \mathcal{S}(G)$ with the minimum number of distinct eigenvalues, $q = q(G)$, and let  $m_1, \ldots, m_q$ be the multiplicities. Necessarily, $n = m_1+ \cdots+ m_q$. On the other hand, by Theorem \ref{thm:mainzp}, the sum $m_i + m_{i+1} + \ldots + m_{i+k-1} \le {Z^{(k+1)}_+}$ for any applicable $i$, and any remaining $m_j \le Z(G)$. Therefore we have,

\begin{eqnarray*}
n &=& m_1 +  \cdots+ m_q \\
&\le&\left  \lfloor \frac{q}{k} \right \rfloor {Z^{(k)}_+(G)} +  (q \% k) Z(G) \\
&=& \frac{q}{k} {Z^{(k)}_+(G)} - \frac{(q \% k)}{k}  {Z^{(k)}_+(G)} +  (q \% k) Z(G) \\
&\le & \frac{q}{k} {Z^{(k)}_+(G)}  +  (q \% k) Z(G) \\
&\le & \frac{q}{k} {Z^{(k)}_+(G)}  +  k Z(G) \\
\end{eqnarray*}
where $(q \% k)$ denotes the remainder of $q$ divided by $k$. Solving for $q$ completes the proof.
\end{proof}

\subsection{\texorpdfstring{IEP-$G$}{IEPG} for graphs on six vertices}
\label{subsec:six}

One of the original motivations of this study was to verify the results from \cite{ordered_iepg} using {\it only} zero forcing parameters. In contrast, \cite{ordered_iepg} utilizes an array of different techniques to narrow down the realizable lists.

As it turns out, using zero forcing parameters on powers of graphs is$\ldots$ powerful$\ldots$ as we are able to remove almost all unobtainable multiplicity lists for graphs on connected 6 vertices. To achieve this, we apply the following results
\begin{itemize}
\item Theorem \ref{thm:main} for $r=1,2,3$, and
\item Theorem \ref{thm:mainzp} for $r=1,2$ using Theorem \ref{thm:pathk23e} regarding $K_{2,3}$ as an induced subgraph where appropriate.
\end{itemize}

Upon implementing this, we discovered that for connected graphs on six vertices, there is no distinction between computing and applying $\Zcheck(\Gamma(G,r))$ 
and $Z(\Gamma(G,r))$  
or even $\Zpcheck(\Gamma(G,r))$ versus $Z_+(\Gamma(G,r))$. However, conceivably, there may be a case, for larger graphs, where  $\Zcheck(\Gamma(G,r)) < Z(\Gamma(G,r))$. However, the computation times for $Z$ and $Z_+$ are substantially faster than $\Zcheck$ and $\Zpcheck$ respectively.

The end result is that the zero forcing parameters are able to narrow down the realizable lists for all but 13 graphs on 6 vertices. The lists that remain from our method are listed in Table \ref{tab:failedsix}. In most all cases, these remaining cases are the same graphs and lists requiring additional analysis or auxiliary results within \cite{ordered_iepg}. In that previous study, the authors utilize previous known results on the minimum number of distinct eigenvalue (e.g., ``$q$'') as well as specialized results which reduces to six exception cases. In contrast, we utilize no prior knowledge on the number of distinct eigenvalues and simply compute zero forcing parameters. In many cases (thought not all), the zero forcing parameters are able to accurately imply the minimum number of distinct eigenvalues correctly. The specialized cases include unicyclic graphs with an odd cycle \cite{minormult}, eigenvalues of trees \cite{JDS03}, the cycle \cite{cycleinverse} or other exceptional cases \cite{ordered_iepg} and are summarized in Table \ref{tab:failedsix}.

\begin{table}[!h]
\begin{tabular}{|l|l|l|} \hline
Graph & Failed Multiplicity Lists & Reason\\ \hline
\texttt{G77}& 1221 & Parter--Wiener Theorem \cite{JDS03}  \\
\texttt{G78}& 1221 & Parter--Wiener Theorem \cite{JDS03}  \\
\texttt{G92}& 2112 & Odd-Unicyclic \cite{minormult} \\
\texttt{G95}& 2112 & Odd-Unicyclic \cite{minormult}  \\
\texttt{G100}& 2112 & Odd-Unicyclic \cite{minormult} \\
\texttt{G104}& 2112 & Odd-Unicyclic \cite{minormult} \\
\texttt{G105}& (2,1,2)1, 1(2,1,2) & Cycle \cite{Ferguson80} \\ 
\texttt{G117}& 132, 213, 231, 312  & Exceptional in \cite{ordered_iepg} \\
\texttt{G121}& 132, 231 & Exceptional in \cite{ordered_iepg}  \\ 
\texttt{G133}& 132, 231 & Exceptional in \cite{ordered_iepg}  \\ 
\texttt{G153}& 312, 213 & Exceptional in \cite{ordered_iepg}  \\
\texttt{G187}& 33 & Wheel, Example \ref{ex:wheel} \\ 
\texttt{G189}& 33 & Previous Results on $q(G)$ (see \cite{BHPRT18}) \\
\hline
\end{tabular}
\caption{A table of the multiplicity lists that Theorems \ref{thm:main}, \ref{thm:mainzp}, and \ref{thm:pathk23e} are unable to rule out. These  multiplicity lists can be ruled out by other methods as cited on the right.}
\label{tab:failedsix}
\end{table}

We remark that the method induced by Theorems~\ref{thm:main}, \ref{thm:mainzp}, and \ref{thm:pathk23e} are able to provide more streamlined certificates for the viable multiplicity lists for \texttt{G125} and \texttt{G138} (as opposed to  \cite{ordered_iepg}) as well as  \texttt{G170}, \texttt{G179} and \texttt{G187} (as opposed to \cite{BHPRT18}).


\section{Skew-Symmetric Matrices}\label{sec:skew}

A skew-symmetric matrix with real entries has $\bA = -\bA\trans$. One basic fact that follows is that all of the eigenvalues of a skew-symmetric matrix are purely imaginary, and in particular, the eigenvalues of $\ii \bA$ are necessarily real. We denote the eigenvalues of  skew-symmetric matrix $\bA$ as $\lambda_i$ with $\im(\lambda_1) \le \cdots \le \im(\lambda_n)$. Note that since the eigenvalues of a matrix with real entries must come in conjugate pairs, we have that $\im( \lambda_{k}) = -\im( \lambda_{n+1-k})$. Since the eigenvalues of $\bA$ can be ordered along the imaginary axis, we can study the ordered eigenvalue multiplicity list problem for skew-symmetric matrices. We will let $m_1, \ldots, m_\ell$ denote the multiplicities of the eigenvalues $\lambda_1,\ldots,\lambda_n$ of $\bA$, and the list $(m_1,\ldots,m_\ell)$ is called the \emph{ordered multiplicity list} of $A$. Indeed, the skew-symmetry leads to additional rules and constraints not present in other cases of the eigenvalue multiplicity list problem.  

For this section regarding skew-symmetric matrices, we will let $Z_-(G)$ denote the skew-forcing number of $G$. As it turns out for a general graph, $Z_-(G) = Z(\gO)$ where $\gO$ is a looped graph (perhaps a multigraph) with no loops.  We will let $\calskew$ denote all $n \times n$ skew-symmetric matrices whose underlying graph is $G$. Lastly, we will generalize the notation $\Gamma(G,r)$ for some integer $r$ into $\Gamma(G,L)$ for some set $L$ of integers.  We define $\Gamma(G, L)$ as a multigraph on the vertex set $V(G)$ such that the number of edges between $i$ and $j $ is the number of (non-lazy) walks from $i$ to $j$ with length in the set $L$.  Therefore, $\Gamma(G,r) = \Gamma(G,\{0,\ldots,r\})$.

To illustrate the difference between the general case and the skew-symmetric case, we have the following.

\begin{lemma}\label{lem.listskew}
Let $G$ be a graph on $n$ vertices, let $\bA \in \calskew$ and let $m_1, \ldots, m_\ell$ be the ordered eigenvalue multiplicity list of $\bA$.

Then,
\begin{enumerate}
\item \label{pal} the list $m_1, m_2 \ldots, m_\ell$ must be palindromic (i.e., the same as its reverse).
\item \label{noncenterskew} for $k \ne \frac{\ell+1}{2}$ ($\ell$ is odd; or any $k$ for $\ell$ even), $m_{k} \le Z(\gloop)$
\item \label{psdskew} for $k=1 \text{ or } \ell$, $m_{k} \le Z_+(G)$.
\item \label{centerskew} if $\ell$ is odd (which is necessarily true if $n$ is odd), then $m_{\frac{\ell+1}{2}} \le Z_-(G)$
\item \label{mirrortwo} for any $k \ne \frac{\ell+1}{2}$,  $m_k + m_{\ell+1-k} \le  Z(\Gamma(G,\{2\})$
\end{enumerate}
\end{lemma}

\begin{proof}
Item \ref{pal} follows from the fact that the eigenvalues of a real skew-symmetric matrix are purely imaginary and must come in conjugate pairs.

From the previous item, $\bA \in \calskew$ has 0 as an eigenvalue if $\ell$ is odd, in which case, the multiplicity of $0$ as an eigenvalue is given by $m_{\frac{\ell+1}{2}}$. Therefore, all other eigenvalues are non-zero, and their multiplicities are the nullity of $\bA - \lambda \bI$, which is bounded by $Z(\gloop)$ since the diagonal entries of $\bA - \lambda \bI$ are all nonzero. Similarly, the multiplicity of 0 is the nullity of $\bA$ which is bounded above by $Z_-(G)$. This gives items \ref{noncenterskew} and \ref{centerskew}.

For item \ref{psdskew}, note that for any matrix $\bA \in \calskew$, $\ii(-A + \lambda_1 I)$ an $\ii(A - \lambda_\ell I)$ are positive semi-definitive Hermitian matrices. It follows from \cite{zf_psd} that $Z_+(G)$ upper bounds $m_1$ and $m_\ell$.

For item \ref{mirrortwo}, since the two corresponding eigenvalues come in conjugate pairs, we can consider the matrix $(\bA - \lambda_1 \bI)(\bA - \overline{\lambda_1} \bI) =
\bA^2 - \lambda_1^2 \bI$ where the quantity $\lambda_1^2$ is necessarily real and negative. The underlying (simple) graph is necessarily an edge configuration of $\Gamma(G,\{2\})$.
\end{proof}

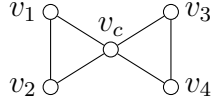
\begin{figure}[h]
    \centering
    \begin{tikzpicture}
    \node[label={above:$v_c$}] (0) at (0,0) {};
    \node[label={right:$v_3$}] (1) at (0.8,0.5) {};
    \node[label={right:$v_4$}] (2) at (0.8,-0.5) {};    
    \node[label={left:$v_1$}] (3) at (-0.8,0.5) {};
    \node[label={left:$v_2$}] (4) at (-0.8,-0.5) {}; 
    \draw (0) -- (1) -- (2) -- (0) -- (3) -- (4) -- (0);
    \end{tikzpicture}
    \caption{Bow-tie graph.}
    \label{fig:bowtie}
\end{figure}

\begin{lemma} \label{lem:bowtie}
Let $G$ be a simple graph. Suppose the induced subgraph $G[W]$ is a bow-tie (as shown in Figure~\ref{fig:bowtie}) for some $B = \{v_c, v_1, v_2, v_3, v_4\} \subseteq V(G)$ such that every path paths connecting the pairs $\{v_1,v_3\}$ and $\{v_2,v_4\}$ of length $3$ are through vertices in $S$. Then, for any $A\in\calskew$, either $\{v_1,v_3\}$ or $\{v_2,v_4\}$ is an edge of the underlying (simple) graph of $A^3-\lambda^2 A$.
\end{lemma}

\begin{proof}
For convenience, we label $v_1,\ldots,v_4$ and $v_c$ as $1,\ldots,4$ and $c$.  
Let $A = \begin{bmatrix}a_{ij}\end{bmatrix}$.  We may replace $A$ by $DAD$ for some signature matrix $D$ and assume that $a_{12}$, $a_{2c}$, $a_{c3}$, and $a_{34}$ are positive.  Since the $1,3$-entry of $A^3$ is 
\[a_{12}a_{2c}a_{c3} + a_{1c}a_{c4}a_{43},\]
we know that $a_{1c}a_{c4}<0$ if $(A^3)_{13} = 0$.  Meanwhile, the $2,4$-entry of $A^3$ is 
\[a_{2c}a_{c3}a_{34} + a_{21}a_{1c}a_{c4},\]
so $a_{1c}a_{c4}>0$ if $(A^3)_{24} = 0$.  Therefore, at least one of $(A^3)_{13}$ and $(A^3)_{24}$ is nonzero.  Since $\lambda^2 A$ contributes nothing to the $1,3$-entry nor the $2,4$-entry, $A^3 - \lambda^2 A$ has at least one nonzero off-diagonal entry. 
\end{proof}

\begin{example}
Consider the bow-tie graph in Figure~\ref{fig:bowtie}. If either of the lists $(2,1,2)$ or $(1,3,1)$ are possible for some $A \in \calskew$ its minimal polynomial is $x(x-\lambda)(x+\lambda) = x^3 - \lambda^2 x$, and so $A^3 - \lambda^2 A$ has nullity $5$ and is equal to $O$. However, Lemma~\ref{lem:bowtie} says that any underlying graph $H$ of $A^3 - \lambda^2 A$  must have an edge, in which case, $Z(H) < 5$, a contradiction. Hence, neither of the the multiplicity lists $(2,1,2)$ nor $(1,3,1)$ are possible.
\end{example}

Recall that the graph $\Gamma(G,\{1,3,5,\ldots,|S|\})$ is the multigraph formed on the vertex set of $G$ and the number of edges between $i$ and $j$ is the number of odd-length walks between them of at most length $|S|$.

\begin{lemma} \label{lem.oddwalk}
Let $A\in\calskew$ and $\lambda_1,\ldots,\lambda_\ell$ its eigenvalues with $\im(\lambda_1)<\cdots<\im(\lambda_\ell)$.  Suppose that $\ell$ is odd and let $c = \frac{\ell+1}{2}$.  Then for any set $S\subseteq \{1,\ldots,\ell\}$ such that $c\in S$ and ${\ell-i+1} \in S$ if and only if $i \in S$, the polynomial
\[p(A) = \prod_{i\in S}(A - \lambda_i I)\]
is skew-symmetric, and $p(A)$ is a matrix in $\mathcal{S}(H)$ for some edge configuration $H$ of $\Gamma(G,\{1,3,5,\ldots,|S|\}$.
\end{lemma}

\begin{proof}
Since $\lambda_j = \lambda_{\ell - j + 1}$, 
\[p(\bA) = \prod_{i\in S}(A - \lambda_i I) 
= A\prod_{\substack{i\in S \\ i < c}}(A^2 - \lambda_i^2 I).\]
Since 
\[p(A)\trans = A\trans\prod_{\substack{i\in S \\ i < c}}((A^2)\trans - \lambda_i^2 I) = -A\prod_{\substack{i\in S \\ i < c}}((-A)^2 - \lambda_i^2 I) = -p(A),\]
the matrix $p(A)$ is a skew-symmetric matrix.
Further, if the $(a,b)$-entry of the matrix $p(\bA)$ is nonzero, then there must be a walk from $a$ to $b$ of odd-length in $G$ as at least one of the odd powers of $\bA$ must have a nonzero $(a,b)$ entry.
\end{proof}

\begin{lemma}\label{lem.oddskew}
Let $G$ be a graph and $A\in\calskew$.  Let $\lambda_1,\ldots,\lambda_\ell$ be the distinct eigenvalues of $A$ with $\im(\lambda_1)<\cdots<\im(\lambda_\ell)$.  Suppose that $\ell$ is odd and let $c = \frac{\ell+1}{2}$.  Then for any set $S\subseteq \{1,\ldots,\ell\}$ such that $c\in S$ and ${\ell-i+1} \in S$ if and only if $i \in S$, 
\[\sum_{j \in S}   m_j\leq Z_-(\Gamma(G,\{1,3,5,\ldots,|S|\}) .\]
\end{lemma}

\begin{proof}
For any $S$ with the given properties, the matrix 
\[p(A) = \prod_{i\in S}(A - \lambda_i I)\]
has nullity $\sum_{j \in S}   m_j$.  By Lemma~\ref{lem.oddwalk}, $p(A)$ is a matrix in $\mathcal{S}(H)$ for some edge configuration of $\Gamma(G,\{1,3,5,\ldots,|S|\}$.  Therefore, 
\[ \sum_{j \in S}   m_j \leq \nul(p(\bA)) \le Z_-(\Gamma(G,\{1,3,5,\ldots,|S|\}),\]
finishing the proof.
\end{proof}

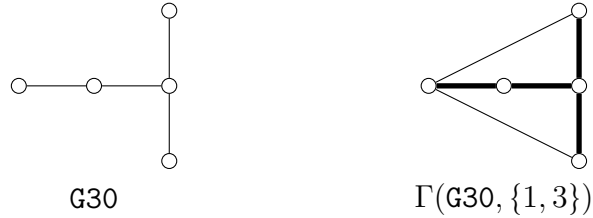
\begin{figure}[h]
    \centering
    \begin{tikzpicture}
    \node (0) at (0,0) {};
    \node (1) at (1,0) {};
    \node (2) at (2,0) {};
    \node (3) at (2,1) {};
    \node (4) at (2,-1) {};
    \draw (0) -- (1) -- (2);
    \draw (3) -- (2) -- (4);
    \node[rectangle,draw=none] at (1,-1.5) {\texttt{G30}};
    \draw[opacity=0] (1,-1.8) -- ++(0.1,0);
    \end{tikzpicture}
    \hfil
    \begin{tikzpicture}
    \node (0) at (0,0) {};
    \node (1) at (1,0) {};
    \node (2) at (2,0) {};
    \node (3) at (2,1) {};
    \node (4) at (2,-1) {};
    \draw[line width=2pt] (0) to (1);
    \draw[line width=2pt] (1) to (2);
    \draw[line width=2pt] (2) to (3);
    \draw[line width=2pt] (2) to (4);
    \draw (3) to (0);
    \draw (0) to (4);
    \node[rectangle,draw=none] at (1,-1.5) {$\Gamma(\mathtt{G30},\{1,3\})$};
    \draw[opacity=0] (1,-1.8) -- ++(0.1,0);
    \end{tikzpicture}
  \caption{The graph \texttt{G30} and the multigraph $\Gamma(\mathtt{G30}, \{1, 3\})$, where the thick edges denote multiedges.}
    \label{fig.examskewgraph1}
\end{figure}

\begin{example}
For an example of an application of Lemma \ref{lem.oddskew}, consider the tree $T$ in Figure~\ref{fig.examskewgraph1}, which is \texttt{G30} in \emph{An Atlas of Graphs} \cite{atlas}. By Lemma~\ref{lem.listskew}, we have that the multiplicities of the nonzero eigenvalues is bounded by $Z(\tloop) = 2$; however, $Z_-(T) = 1$, so the multiplicity of $0$ is bounded by $1$. As a result, Lemma~\ref{lem.listskew}, $(1,3,1)$ is not possible. Leaving two possible skew eigenvalue multiplicity lists: $(1,1,1,1,1)$ and $(2,1,2)$. 

However, we now discount $(2,1,2)$ using Lemma~\ref{lem.oddskew}. Observe that there are two pairs of vertices that are exactly distance 3 in $T$. In particular, it is not possible to realize the empty graph in $\Gamma(G, \{1,3\})$, and therefore, $Z_-(\Gamma(G, \{1,3\})<5$. Hence, by Lemma \ref{lem.oddskew} the list $(2,1,2)$ is not possible.
\end{example}

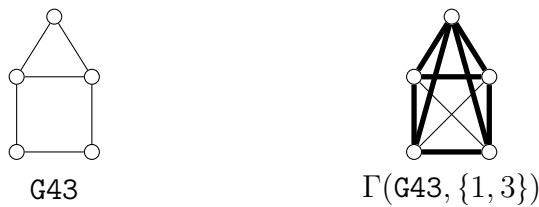
\begin{figure}[h]
    \centering
    \begin{tikzpicture}
    \node (0) at (0,0) {};
    \node (1) at (1,0) {};
    \node (2) at (0,1) {};
    \node (3) at (1,1) {};
    \node (4) at (0.5,1.8) {};
    \draw (0) -- (1) -- (3) -- (2) -- (0);
    \draw (2) -- (4) -- (3);
    \node[rectangle,draw=none] at (0.5,-0.5) {\texttt{G43}};
    \draw[opacity=0] (0.5,-0.8) -- ++(0.1,0);
    \end{tikzpicture}
    \hfil    
    \begin{tikzpicture}
    \node (0) at (0,0) {};
    \node (1) at (1,0) {};
    \node (2) at (0,1) {};
    \node (3) at (1,1) {};
    \node (4) at (0.5,1.8) {};
    \draw (0) -- (3);
    \draw (2) -- (1);
    \draw[line width=2pt] (0) to (1);
    \draw[line width=2pt] (0) to (2);
    \draw[line width=2pt] (0) to (4);
    \draw[line width=2pt] (1) to (3);
    \draw[line width=2pt] (1) to (4);
    \draw[line width=2pt] (2) to (3);
    \draw[line width=2pt] (2) to (4);
    \draw[line width=2pt] (3) to (4);
    \node[rectangle,draw=none] at (0.5,-0.5) {$\Gamma(\mathtt{G43},\{1,3\})$};
    \draw[opacity=0] (0.5,-0.8) -- ++(0.1,0);
    \end{tikzpicture}
    \caption{The graph \texttt{G30} and the multigraph $\Gamma(\mathtt{G43}, \{1, 3\})$, where the thick edges denote multiedges.}
    \label{fig.examskewgraph2}
\end{figure}

\begin{example}
Consider the graph $\mathtt{G45}$ in Figure~\ref{fig.examskewgraph2}.  We have $Z(\mathtt{G45}) = 2$ and $Z_-(\Gamma(\mathtt{G45},\{1,3\})) = 4$. By Lemma $\ref{lem.oddskew}$, the sum of three eigenvalue multiplicity (including the one for $0$) is at most 4. Therefore, $(2,1,2)$ and $(1,3,1)$ are not possible, and only $(1,1,1,1,1)$ is possible.
\end{example}

\subsection{Skew \texorpdfstring{IEP-$G$}{IEPG} on graphs with five vertices}

Just as with the IEP-$G$ on six vertices, we can apply our techniques to get a head start on the IEP-$G$ for skew-symmetric matrices on five vertices. We apply Lemma~\ref{lem.listskew}, Lemma~\ref{lem:bowtie}, and Lemma~\ref{lem.oddskew} to determine all possible multiplicity lists.

This method is able to determine all but three realizable multiplicity lists. Two of which are for the tree on five vertices. The complete table of realizable skew multiplicity lists can be found in \ref{tab:fivelists} and corresponding matrices are in  \ref{sec:skewrealize}.

\begin{theorem}[\cite{matchingskew}] 
\label{thm:matchingskew}
Let $G$ be a graph and let $m'(G)$ be the matching number of $G$. Then, the maximum rank over all matrices in $\calskew$ is exactly $2m'(G)$
\end{theorem}

In particular, for a star (i.e., \texttt{G29}), $m'(T) = 1$, so the maximum rank is 2. As a result, the sum of all non-central multiplicities is 2, so $(1,1,1,1,1)$ and $(2,1,2)$ are not possible, and any multiplicity list must be $(1,3,1)$.

The other exceptional care is $K_{2,3} + e$ (see \ref{fig:k23e}, \texttt{G46}.

\begin{proposition} \label{prop:g46impossible}
For {\rm\texttt{G46}} the skew multiplicity list $(2,1,2)$ is not feasible. 
\end{proposition}

\begin{proof}
Let $G$ be the graph \texttt{G46}.  Suppose $A$ is a skew-symmetric matrix in $\calskew$ with its ordered multiplicity list $(2,1,2)$.  By replacing $A$ with $\frac{1}{\im(\lambda)}A$ if necessary, we may assume the spectrum of $A$ is $\{-\ii^{(2)}, 0, \ii^{(2)}\}$.  Let 
\[A = \begin{bmatrix} A_{11} & A_{12} \\ A_{21} & O \end{bmatrix},\]
where $A_{12} = -A_{21}\trans$.  

If $\rank(A_{21}) = 1$, then $\rank(A)\leq 3$, violating the fact that $0$ only has multiplicity $1$.  Hence we assume $\rank(A_{12}) = 2$.  Suppose $\bx =\begin{bmatrix} x_1 & x_2 & x_3 & x_4 & x_5\end{bmatrix}\trans$ is a vector such that $A\bx = \bzero$.  Then $x_1 = x_2 = 0$ since  
\[\begin{bmatrix} A_{12} & O \end{bmatrix}\bx = \bzero\]
and $A_{12}$ has full column-rank.  Thus, the first two entries of any vector in the kernel of $A$ is zero.  

By the spectrum of $A$, 
\[O = A(A + \ii I)(A - \ii I) = A(A^2 + I),\]
so the columns of $A$ are vectors in the kernel of $A$.  By the observation on vectors in the kernel of $A$ and the fact that $A^2 + I$ is symmetric, $A^2$ has the form 
\[\begin{bmatrix} -I_2 & O_{2,3} \\
 O_{3,2} & ? \end{bmatrix}.\]
 This means $-\ba_1\trans\ba_3 = 0$, where $\ba_j$ is the $j$-th column of $A$.  However, this is impossible since there is only one index, namely $2$, where both $\ba_1$ and $\ba_3$ are nonzero.
 \end{proof}

All remaining skew multiplicity lists can be found in the appendix with their realizations.

\section{Conclusion and Future Considerations}

The techniques developed and used in this article seem very promising, and we believe the results within here may just be the tip of the iceberg. We briefly pose possible directions for future considerations.

In Section~\ref{sec:power} it is mentioned that the number of distinct eigenvalues, $q(G)$ has been proven by Kim and Shader to be as large as $\frac98 |V(G)|$,  \cite{kim2009smith}. However, it has been speculated that $q(G)$ may be superlinear if not exponential in $|V(G)|$. We would hope that the method of zero forcing on powers of graphs, would shine brighter line on this problem.

In Section \ref{subsec:six}, it appears as though $\Zcheck^{(r)}(G)$ and $Z(\Gamma(G,2))$ are equal for small graphs. Indeed, for the classical zero forcing parameter $Z(G) = \Zhat(G) = M(G)$ for all graphs up to 7 vertices. However, it is not clear if this would hold for larger graphs. It would be interesting to find an example where equality does not hold.

Previous work on zero forcing and eigenvalue multiplicities was considered in \cite{rigid} using a newly defined variation of zero forcing: rigid linkage forcing. We ask: {\it Is there a concrete relationship between relation rigid linkage forcing and the zero forcing numbers of power of graphs?}

\bibliographystyle{plain}
\bibliography{Zpower-bib}

\appendix

\section{Ordered Multiplicity Lists for Skew-Symmetric Matrices on Graphs with 5 vertices} \label{sec:skewrealize}

Here, we provide all skew symmetric matrices that realize all possible multiplicity lists. With the exception of \texttt{G29} (the star graph on 5 vertices), (1,1,1,1,1) is possible by choosing a near-arbitrary matrix in $\calskew$. Hence, all cases of $(1,1,1,1,1)$ are omitted.

\begin{table}
\begin{center}
\begin{tabular}{|l|l|l|} \hline
Graph & Skew-Multiplicity lists\\ \hline
\texttt{G29}& (1,3,1) , \st{(1,1,1,1,1)}\\
\texttt{G30}& (1,1,1,1,1)\\
\texttt{G31}& (1,1,1,1,1)\\
\texttt{G34}& (1,1,1,1,1)\\
\texttt{G35}& (1,1,1,1,1)\\
\texttt{G36}& (1,1,1,1,1)\\
\texttt{G37}& (1,1,1,1,1)\\ 
\texttt{G38}& (1,1,1,1,1)\\
\texttt{G40}& (1,1,1,1,1)\\
\texttt{G41}& (1,1,1,1,1)\\
\texttt{G42}& (1,1,1,1,1)\\ 
\texttt{G43}& (1,1,1,1,1)\\
\texttt{G44}& (1,3,1) , (2,1,2) , (1,1,1,1,1)\\
\texttt{G45}& (2,1,2) , (1,1,1,1,1)\\
\texttt{G46}& (1,3,1) , \st{(2,1,2)} , (1,1,1,1,1)\\
\texttt{G47}& (2,1,2) , (1,1,1,1,1)\\
\texttt{G48}& (2,1,2) , (1,1,1,1,1)\\
\texttt{G49}& (2,1,2) , (1,1,1,1,1)\\
\texttt{G50}& (1,3,1) , (2,1,2) , (1,1,1,1,1)\\
\texttt{G51}& (1,3,1) , (2,1,2) , (1,1,1,1,1)\\
\texttt{G52 }& (1,3,1) , (2,1,2) , (1,1,1,1,1) \\ \hline
\end{tabular}
\end{center}
\caption{Table of all possible skew multiplicity lists for connected graphs on 5 vertices. \st{Strikethrough} lists are not feasible but require auxiliary results to our methods.}
\label{tab:fivelists}
\end{table}

\begin{itemize}
\item[]  \texttt{G42}, (2,1,2): not possible by Theorem~\ref{thm:matchingskew}.

\item[] \texttt{G44}, (1,3,1):
\[ \left(
\begin{array}{ccccc}
 0 & 0 & -1 & -1 & -1 \\
 0 & 0 & -1 & -1 & -1 \\
 1 & 1 & 0 & 0 & 0 \\
 1 & 1 & 0 & 0 & 0 \\
 1 & 1 & 0 & 0 & 0 \\
\end{array}
\right) \]

\item[] \texttt{G44}, (2,1,2):

\[ \left(
\begin{array}{ccccc}
 0 & 0 & 1 & 1 & 1 \\
 0 & 0 & \frac{1}{\sqrt{2}} & \sqrt{2} & -\sqrt{2} \\
 -1 & -\frac{1}{\sqrt{2}} & 0 & 0 & 0 \\
 -1 & -\sqrt{2} & 0 & 0 & 0 \\
 -1 & \sqrt{2} & 0 & 0 & 0 \\
\end{array}
\right) \]

\item[] \texttt{G45}: (2,1,2):

\[ \left(
\begin{array}{ccccc}
 0 & 1 & 2 & 1 & 0 \\
 -1 & 0 & 2 & -1 & 0 \\
 -2 & -2 & 0 & \frac{1}{2} & 0 \\
 -1 & 1 & -\frac{1}{2} & 0 & -\frac{3 \sqrt{3}}{2} \\
 0 & 0 & 0 & \frac{3 \sqrt{3}}{2} & 0 \\
\end{array}
\right)
\] 
\item[] \texttt{G46} (2,1,2): is not possible by Proposition \ref{prop:g46impossible}.

\noindent and \texttt{G46} (1,3,1):

\[ \left(
\begin{array}{ccccc}
 0 & 0 & 0 & 1 & 1 \\
 0 & 0 & 0 & 1 & 1 \\
 0 & 0 & 0 & 1 & 1 \\
 -1 & -1 & -1 & 0 & 1 \\
 -1 & -1 & -1 & -1 & 0 \\
\end{array}
\right) \]

\item[] \texttt{G47} (2,1,2):

\[ \left(
\begin{array}{ccccc}
 0 & 1 & 0 & 0 & 1 \\
 -1 & 0 & -1 & -\frac{1}{\sqrt{3}} & -\frac{1}{\sqrt{3}} \\
 0 & 1 & 0 & -1 & 0 \\
 0 & \frac{1}{\sqrt{3}} & 1 & 0 & -\frac{2}{\sqrt{3}} \\
 -1 & \frac{1}{\sqrt{3}} & 0 & \frac{2}{\sqrt{3}} & 0 \\
\end{array}
\right)\]

\item[] \texttt{G47} (2,1,2):

\[ \left(
\begin{array}{ccccc}
 0 & 1 & 0 & 0 & -\sqrt{2} \\
 -1 & 0 & \frac{3}{2} & 0 & \frac{1}{2} \\
 0 & -\frac{3}{2} & 0 & 1 & -\frac{1}{\sqrt{2}} \\
 0 & 0 & -1 & 0 & 1 \\
 \sqrt{2} & -\frac{1}{2} & \frac{1}{\sqrt{2}} & -1 & 0 \\
\end{array}
\right) \]

\item[] \texttt{G48} (2,1,2):

\[ \left(
\begin{array}{ccccc}
 0 & 0 & 1 & 1 & 1 \\
 0 & 0 & \frac{1}{2} & -2 & \frac{1}{2} \\
 -1 & -\frac{1}{2} & 0 & 0 & -\sqrt{\frac{5}{2}} \\
 -1 & 2 & 0 & 0 & 0 \\
 -1 & -\frac{1}{2} & \sqrt{\frac{5}{2}} & 0 & 0 \\
\end{array}
\right)
\] 

\item[] \texttt{G49} (2,1,2):

\[ \left(
\begin{array}{ccccc}
 0 & 0 & -\frac{3 \sqrt{3}}{2} & 1 & 1 \\
 0 & 0 & 0 & 1 & -2 \\
 \frac{3 \sqrt{3}}{2} & 0 & 0 & 1 & -\frac{1}{2} \\
 -1 & -1 & -1 & 0 & -\sqrt{3} \\
 -1 & 2 & \frac{1}{2} & \sqrt{3} & 0 \\
\end{array}
\right) \] 

\item[] \texttt{G50} (2,1,2):

\[ \left(
\begin{array}{ccccc}
 0 & 1 & 0 & 1 & 1 \\
 -1 & 0 & 1 & 0 & 1 \\
 0 & -1 & 0 & -1 & 1 \\
 -1 & 0 & 1 & 0 & -1 \\
 -1 & -1 & -1 & 1 & 0 \\
\end{array}
\right) \]

\item[] \texttt{G50} (1,3,1): \[  \left(
\begin{array}{ccccc}
 0 & 1 & 0 & 1 & 1 \\
 -1 & 0 & 1 & 0 & 1 \\
 0 & -1 & 0 & -1 & -1 \\
 -1 & 0 & 1 & 0 & 1 \\
 -1 & -1 & 1 & -1 & 0 \\
\end{array}
\right) \]

\item[] \texttt{G51} (2,1,2):

\[ \left(
\begin{array}{ccccc}
 0 & 1 & 0 & 1 & 1 \\
 -1 & 0 & 1 & 1 & -1 \\
 0 & -1 & 0 & -\frac{1}{3} & -\frac{1}{3} \\
 -1 & -1 & \frac{1}{3} & 0 & \frac{4}{3} \\
 -1 & 1 & \frac{1}{3} & -\frac{4}{3} & 0 \\
\end{array}
\right) \]

\noindent and \texttt{G51} (1,3,1):

\[ \left(
\begin{array}{ccccc}
 0 & 0 & 1 & 1 & -1 \\
 0 & 0 & 1 & 1 & -1 \\
 -1 & -1 & 0 & 1 & 1 \\
 -1 & -1 & -1 & 0 & 2 \\
 1 & 1 & -1 & -2 & 0 \\
\end{array}
\right) \]

\item[] \texttt{G52} (2,1,2):

\[
\left(
\begin{array}{ccccc}
 0 & a & b & c & 1 \\
 -a & 0 & 1 & 1 & 1 \\
 -b & -1 & 0 & 1 & 1 \\
 -c & -1 & -1 & 0 & 1 \\
 -1 & -1 & -1 & -1 & 0 \\
\end{array}
\right)
\]

where $a\approx -0.249038$ the negative root to $x^4-4x-1$, $b \approx 1.35219$ is the positive root to $x^4 +2x^2 -7$, and $c\approx-1.66325$ is the negative root to $x^4+4x-1$.

\noindent and \texttt{G52} (1,3,1):

\[ \left(
\begin{array}{ccccc}
 0 & 1 & 1 & \frac{1}{2} & -1 \\
 -1 & 0 & 1 & 1 & 1 \\
 -1 & -1 & 0 & \frac{1}{2} & 2 \\
 -\frac{1}{2} & -1 & -\frac{1}{2} & 0 & \frac{3}{2} \\
 1 & -1 & -2 & -\frac{3}{2} & 0 \\
\end{array}
\right) \]
\end{itemize}

\end{document}